\documentclass[a4paper,12pt]{amsart}

\usepackage{amssymb, enumitem,mathtools}
\usepackage{xcolor}

\usepackage[breaklinks=true,colorlinks=true,linkcolor=green,citecolor=red,urlcolor=blue]{hyperref}

\allowdisplaybreaks

\newcommand \on{\overline{\nabla}}
\newcommand \G{\Gamma}

\newcommand \la{\lambda}

\newcommand \br{\mathbb{R}}

\newcommand \id{\operatorname{id}}
\newcommand \<{\langle}
\renewcommand \>{\rangle}
\newcommand \ip{\< \cdot, \cdot \>}

\newcommand \mU{\mathcal{U}}

\newcommand \Tr{\operatorname{Tr}}

\newcommand \tM{\overline{M}}

\newcommand \tR{\overline{R}}

\theoremstyle{plain}
\newtheorem{theorem}{Theorem}
\newtheorem*{theorem*}{Theorem}
\newtheorem*{corollary*}{Corollary}

\newtheorem*{conj*}{Conjecture}
\newtheorem{lemma}{Lemma}
\newtheorem*{lemma*}{Lemma}

\newtheorem*{prop*}{Proposition}

\theoremstyle{definition}

\newtheorem*{definition*}{Definition}

\theoremstyle{remark}

\newtheorem{example}{Example}

\makeatletter
\@namedef{subjclassname@2020}{%
  \textup{2020} Mathematics Subject Classification}
\makeatother

\begin{document}

\title{Weakly Einstein hypersurfaces in space forms}

\author{Jihun Kim}
\address{Department of Mathematics, Sungkyunkwan University, Suwon 16419, Korea, and\newline
\indent Department of Mathematical and Physical Sciences, La Trobe University, Melbourne, Victoria, 3086, Australia}
\email{jihunkim@skku.edu, jihun.kim@latrobe.edu.au}

\author{Yuri Nikolayevsky}
\address{Department of Mathematical and Physical Sciences, La Trobe University, Melbourne, Victoria, 3086, Australia}
\email{y.nikolayevsky@latrobe.edu.au}

\author{JeongHyeong Park}
\address{Department of Mathematics, Sungkyunkwan University, Suwon 16419, Korea}
\email{parkj@skku.edu}

\thanks {The first and third named authors were supported by the National Research Foundation of Korea(NRF) grant funded by the Korea government(MSIT) (RS-2024-00334956).
\\
\indent The first named author is thankful to the Department of Mathematical and Physical Sciences at La Trobe University for hospitality.
\\
\indent The second named author was partially supported by ARC Discovery grant DP210100951.
}
\subjclass[2020]{Primary 53C25, 53B25} 
\keywords{weakly Einstein manifold, submanifold in a space form}


\dedicatory{Dedicated to the memory of Professor Old\v{r}ich Kowalski}

\begin{abstract}
A Riemannian manifold $(M,g)$ is called \emph{weakly Einstein} if the tensor $R_{iabc}R_{j}^{~~abc}$ is a scalar multiple of the metric tensor $g_{ij}$. We give a complete classification of weakly Einstein hypersurfaces in the spaces of nonzero constant curvature (the classification in a Euclidean space has been previously known). The main result states that such a hypersurface can only be the product of two spaces of constant curvature or a rotation hypersurface.
\end{abstract}

\maketitle

\section{Introduction}
\label{s:intro}

Einstein hypersurfaces in spaces of constant curvature are well known: by \cite[Theorem~7.1]{Fia}, any such (connected) hypersurface is either totally umbilical (in particular, totally geodesic), or is of conullity $1$, or is an open domain of the product of two spheres of the same Ricci curvature in the sphere. In this paper, we study hypersurfaces in the spaces of constant curvature which satisfy a similar, but different intrinsic condition, as in the following definition.

\begin{definition*}
    An $n$-dimensional Riemannian manifold $(M,g)$ is called \emph{weakly Einstein} if
    \begin{equation*}
        \check{R}_{ij}:=R_{iabc}R_{j}^{~~abc}=\frac{\|R\|^2}{n}g_{ij}.
    \end{equation*}
\end{definition*}

The definition of weakly Einstein manifolds was coined by Euh, Park, and Sekigawa, who provided several applications, in their study of a curvature identity that holds on any 4-dimensional Riemannian manifold \cite{EPS1, EPS2, EPS3}. Clearly, any $2$-dimensional manifold is weakly Einstein (so throughout the paper we always assume that $n \ge 3$). A $3$-dimensional weakly Einstein manifold is either of constant curvature, or has its Ricci operator of rank one~\cite[Lemma~5]{GHMV}. In dimension $4$, any Einstein manifold is weakly Einstein by \cite{EPS1} (which was the motivation for the definition). Similarly to Einstein manifolds, weakly Einstein \emph{compact} manifolds admit a variational definition: they are critical points of the integral of the squared norm of the curvature tensor on the space of metrics of fixed volume with parallel Ricci tensor. However, the reader should not be misled by the terminology here. In general, an Einstein Riemannian manifold does not have to be weakly Einstein, with the simplest example being the product of two non-isometric round spheres of the same Ricci curvature (although in a class of manifolds possessing a specific structure, these two conditions may be equivalent \cite[Corollary 2.3.1]{KPS}). Moreover, there is no Schur-type result for weakly Einstein manifolds (the proportionality function $\|R\|^2$ in the definition above does not have to be constant). However, if a manifold is weakly Einstein \emph{and} Einstein (or more generally, if its Ricci tensor is Codazzi), then $\|R\|^2$ must be constant when $n \ne 4$ \cite[Lemma~3.3]{BV}.

There is a number of results on weakly Einstein manifolds known in the literature. Arias-Marco and Kowalski classified weakly Einstein $4$-dimensional homogeneous Riemannian manifolds~\cite{AMK}. Several results on weakly Einstein manifolds are given in \cite[\S 6.55-6.63]{Be}. Conformally flat weakly Einstein manifolds are classified in \cite[Theorem~2]{GHMV}. Any irreducible symmetric space is weakly Einstein (as the isotropy representation is irreducible); a reducible symmetric space is weakly Einstein if the factors have the same proportionality constant $\|R\|^2/n$.  Note that some authors include the requirement of not being Einstein in the definition of weakly Einstein manifold. We do not do that and will always say explicitly when the non-Einstein condition is being imposed.

Not too many results on weakly Einstein manifolds are known in the context of submanifold geometry. By~\cite[Theorem~12]{GHMV}, a non-Einstein, weakly Einstein hypersurface $M^n$ in a Euclidean space $\br^{n+1}, \; n \ge 2$, is locally isometric to a warped product of an interval of the real line and a space of constant curvature, with a very special warping function. Extrinsically, such a hypersurface has two principal curvatures $\la \ne 0$ and $-\la$, with multiplicities $1$ and $n-1$, respectively, and is a \emph{rotation} hypersurface (see~\cite[Theorem~4.2]{dCD}). Moreover, from~\cite[Proposition~3.2]{dCD} one can deduce that the hypersurface $M^n$ is homothetic to a part of a \emph{generalised catenoid}: relative to Cartesian coordinates $x_1, \dots, x_{n+1}$ in $\br^{n+1}$, it is the image of the catenary $x_{n+1} = \cosh x_1$ under the action of the subgroup $\mathrm{SO}(n) \subset \mathrm{SO}(n+1)$ keeping the $x_1$-axis pointwise fixed (earlier, this fact appeared in a conference presentation by Mari\~{n}o-Villar). In~\cite{WZ1, WZ2}, the authors classified weakly Einstein hypersurfaces with constant $\|R\|^2$ in $\mathbb{C}P^2$ and $\mathbb{C}H^2$.

As we mentioned earlier, by \cite[Theorem~7.1]{Fia}, an Einstein hypersurface of dimension $n \ge 4$ of a space of constant curvature must either be totally umbilical, or have zero extrinsic curvature, or be a product of two spheres of the same Ricci curvature in the sphere (see Example~\ref{ex:prod}). In the first two cases (and also in the cases when $n \le 3$), the resulting hypersurface is of constant curvature, and hence is weakly Einstein. In the third case, the hypersurface is Einstein only when it is the product of two congruent spheres.

To state the classification, we introduce two classes of weakly Einstein hypersurfaces (both these classes have already been studied in the literature, see~\cite{MV, GHMV}).

\begin{example} \label{ex:prod}
  \emph{Products}. It is easy to see from the definition that the Riemannian product $M^n=M^{n_1}_{1}(\kappa_1)\times M^{n_2}_{2}(\kappa_2)$ of two spaces of constant curvature is weakly Einstein if and only if $\kappa_1^2(n_1-1)=\kappa_2^2(n_2-1)$ (and is Einstein if and only if $\kappa_1 (n_1-1)=\kappa_2 (n_2-1)$). This gives the following construction. Let $n_2 > n_1 \ge 2$, and let $\br^{n_1+1}$ be equipped with a Euclidean inner product $g_1$, and $\br^{n_2+1}$ be equipped with either Euclidean or Minkowski inner product $g_2$. For $i=1,2$, let $M^{n_i}(\kappa_i) \subset \br^{n_i+1}$ be a connected component of the subset $g_i(x,x) = \kappa_i^{-1}, \; x \in \br^{n_i+1}$, with $\kappa_1^2(n_1-1) = \kappa_2^2(n_2-1)$. Then the Cartesian product $M^{n_1}(\kappa_1) \times M^{n_2}(\kappa_2) \subset \br^{n_1+n_2+2}$ is a weakly Einstein, non-Einstein hypersurface in the space of constant curvature $M^{n_1+n_2+1}(\kappa) \subset \br^{n_1+n_2+2}$, where $\kappa^{-1}= \kappa_1^{-1} + \kappa_2^{-1}$.
\end{example}

\begin{example} \label{ex:rotation}
  \emph{Rotation hypersurfaces}. We follow the construction in~\cite[Section~2]{dCD}. Let $\br^3$ be a linear space with the Cartesian coordinates $x_1, x_2$ and $x_3$ equipped with either Euclidean or Minkowski inner product $g_1$, and let $\br^{n-1}$ be equipped with a Euclidean inner product $g_2$. Let $g=g_1+g_2$ be the inner product on $\br^{n+2} = \br^3 \oplus \br^{n-1}$. Define the surface $N^2(\kappa) \subset (\br^3, g_1)$ of constant Gauss curvature $\kappa$, where $\kappa > 0$ (respectively, $\kappa < 0$) if $g_1$ is Euclidean (respectively, Minkowskian), to be a connected component of the set $g_1(x,x) = \kappa^{-1}$. Let $G$ be the identity component of the group of isometries of $(\br^{n+2},g)$ which keeps the points of the $(x_2,x_3)$-plane pointwise fixed. Then the orbit $\tM^{n+1}(\kappa) \subset \br^{n+2}$ of $N^2(\kappa)$ under the action of $G$ is a hypersurface of constant curvature. If $\gamma \subset N^2(\kappa)$ is a curve, the orbit $\mathrm{R}_\gamma \subset \tM^{n+1}(\kappa)$ under the action of $G$ is called the \emph{rotation hypersurface with profile $\gamma$}.

  Let $\gamma(s)=(x_1(s),x_2(s),x_3(s))$ be a curve lying on the surface $N^2(\kappa) \subset \br^3$, where $s$ is an arc length parameter, so that $g_1(\gamma(s),\gamma(s)) = \kappa^{-1}$ and $g_1(\dot{\gamma}(s),\dot{\gamma}(s)) = 1$. We consider the following four choices for the curve $\gamma$ (depending on one real parameter) and the inner product $g_1$ (where $x=(x_1,x_2,x_3)$).
  \begin{enumerate}[label={\rm{(\roman*)}}, ref=\roman*]
    \item \label{it:rotss}
    Spherical case. $g_1(x,x) = x_1^2 +x_2^2+x_3^2$ and $x_1^2 = s^2 + \kappa^{-1} \eta^2$, where $0 < \eta <1$ and $s \in (-\sqrt{\kappa^{-1}\eta(1-\eta)},\sqrt{\kappa^{-1}\eta(1-\eta)})$.

    \item \label{it:rothh}
    Hyperbolic-hyperbolic case. $g_1(x,x) = -x_1^2 +x_2^2+x_3^2$ and $x_1^2 = -s^2 - \kappa^{-1} \eta^2$, where $\eta > 1$ and $s \in (-\sqrt{-\kappa^{-1}\eta(\eta-1)},\sqrt{-\kappa^{-1}\eta(\eta-1)})$.

    \item \label{it:roths}
    Hyperbolic-spherical case. $g_1(x,x) = x_1^2 +x_2^2-x_3^2$ and $x_1^2 = s^2 + b$, where either $b > \kappa^{-1}$ and then $s \in (-\infty, +\infty)$ or $b= \kappa^{-1} \eta^2, \; \eta > 0$ and $s \in (\sqrt{-\kappa^{-1}\eta(1+\eta)},+\infty)$.

    \item \label{it:rothp}
    Hyperbolic-parabolic case. $g_1(x,x) = 2x_1x_3 +x_2^2$ and $x_1^2 = as$, where $a > 0$ and $s \in (\frac12\sqrt{-\kappa^{-1}},+\infty)$.
  \end{enumerate}
  The equations of the curve $\gamma$ can be presented explicitly in terms of elliptic integrals. 

  In all four cases (\ref{it:rotss} -- \ref{it:rothp}), the rotation hypersurface $\mathrm{R}_\gamma \subset \tM^{n+1}(\kappa)$ is weakly Einstein and not Einstein. This can be easily seen from~\cite[Proposition~3.2]{dCD} and~\eqref{eq:cRaa}.
\end{example}

Note that the hypersurfaces in Examples~\ref{ex:prod} and~\ref{ex:rotation} have exactly two principal curvatures. 

We prove the following.

\begin{theorem} \label{th:main}
  Let $M^n \subset \tM^{n+1}(\kappa), \; n \ge 3$, be a connected, weakly Einstein, non-Einstein smooth hypersurface in the space of constant curvature $\kappa \ne 0$. Then one of the following holds.
  \begin{enumerate}[label={\rm{(\Alph*)}},ref=\Alph*]
    \item\label{it:thprod}
    The hypersurface $M^n \subset \tM^{n+1}(\kappa)$ is the product hypersurface as given in Example~\ref{ex:prod} \emph{(}note that we must have $n \ge 5$\emph{)}.

    \item\label{it:throt}
    The hypersurface $M^n \subset \tM^{n+1}(\kappa)$ is the rotation hypersurface as given in Example~\ref{ex:rotation}.
  \end{enumerate}
\end{theorem}

Note that the smoothness condition in Theorem~\ref{th:main} can be relaxed to $C^4$ regularity.

The fact that a weakly Einstein hypersurface having two principal curvatures in a non-flat space of constant curvature is either the product of two spaces of constant curvature or is a rotation hypersurface is proved in~\cite[Lemma~3.12]{MV}.

\section{Preliminaries}
\label{s:pre}

Let $M^n \subset \tM^{n+1}(\kappa)$ be a connected weakly Einstein hypersurface in the space of constant curvature $\kappa \ne 0$. Denote $\ip$ the metric tensor on $\tM^{n+1}(\kappa)$ and the induced metric tensor on $M^n$. Let $\on, \tR$ and $\nabla, R$ denote the Levi-Civita connections and the curvature tensors of $\tM^{n+1}$ and of $M^n$, respectively. Let $\xi$ be a unit normal vector field of $M^n$. For $x \in M^n$ we define the shape operator $S$ on $T_xM^n$ by $SX=-\on_X \xi$, so that $\<\on_XY,\xi\>=\<SX,Y\>$, where $X, Y \in T_xM^n$.

For $x \in M^n$, let $e_i, \; i=1, \dots, n$, be an orthonormal basis for $T_xM^n$ of eigenvectors of $S$, with $\la_i,\; i=1, \dots, n$, being the corresponding eigenvalues. From Gauss equations, the only nonzero components of the curvature tensor $R$ of $M^n$ at $x$ (up to the symmetries) are given by $R(e_i,e_j,e_j,e_i)=\kappa + \la_i\la_j,\; i \ne j$.
It follows that $\check{R}(e_i,e_j) = 0$ for $i \ne j$, and computing $\check{R}(e_i,e_i)$ we find that weakly Einstein condition is equivalent to the equations
\begin{equation}\label{eq:cRii}
  \sum_{j \ne i} (\kappa+\la_i\la_j)^2 = \tfrac{1}{n}\|R\|^2, \quad \text{for } i=1, \dots, n.
\end{equation}

Denote $\mu_1, \dots, \mu_q$ pairwise non-equal eigenvalues of $S$ (so that the $\mu_a$'s are the $\la_i$'s without repetitions), with multiplicities $p_1, \dots, p_q$, respectively. We have $p_a \ge 1$ and $\sum_{a=1}^{q}p_a=n$. In terms of the $\mu_a$'s, equations~\eqref{eq:cRii} take the following form:
\begin{equation}\label{eq:cRaa}
  (p_a-1) (\kappa+\mu_a^2)^2 + \sum_{b \ne a} p_b (\kappa+\mu_a\mu_b)^2 = \tfrac{1}{n}\|R\|^2, \quad \text{for } a=1, \dots, q.
\end{equation}

The following fact is proved in~\cite[Remark~14]{GHMV}.

\begin{lemma} \label{l:four}
In the notation as above, at every point $x \in M^n$, all the principal curvatures $\la_i, \; i=1, \dots, n$, \emph{(}equivalently, all $\mu_a, \; a=1, \dots, q$\emph{)} satisfy the equation
\begin{equation}\label{eq:la4}
  t^4 + (2\kappa-H_2)t^2 - 2\kappa nHt + H_3=0,
\end{equation}
where $H=\frac{1}{n}\Tr S, \; H_2=\Tr (S^2)$ and $H_3=\tfrac{1}{n}\|R\|^2-(n-1) \kappa^2$. In particular, the number $q$ of principal curvatures satisfies $q \le 4$.
\end{lemma}
\begin{proof}
  Equation~\eqref{eq:la4} easily follows from~\eqref{eq:cRii}. The fact that $q \le 4$ follows from~\eqref{eq:la4}: all the principal curvatures satisfy the same polynomial equation of degree $4$.
\end{proof}

Consider an open, connected domain $\mU$ of $M^n$ on which the number of principal curvatures and their multiplicities are constant (in our notation, both $q$ and $p_a,\, a=1,\dots, q$, are constant). On the domain $\mU$, the functions $\la_i$ (and $\mu_a$) are smooth, and we can locally choose an orthonormal frame $\{e_1, \dots, e_n\}$ of eigenvectors of the shape operator $S$.

Denote $\G_{ki}^{\hphantom{k}j}=\<\nabla_{e_k} e_i, e_j\>$ (note that $\G_{kj}^{\hphantom{k}i} = -\G_{ki}^{\hphantom{k}j}$). Codazzi equation takes the form
  \begin{equation}\label{eq:codazzi}
    e_k(\la_i)=\G_{ii}^{\hphantom{i}k} (\la_i-\la_k), \qquad (\la_i-\la_j) \G_{kj}^{\hphantom{k}i} = (\la_k-\la_j) \G_{ij}^{\hphantom{i}k},
  \end{equation}
for pairwise nonequal $i,j,k =1, \dots, n$.
Gauss equation gives
\begin{equation}\label{eq:gauss}
\begin{split}
  R_{kijs}:=\<R(e_k,e_i)e_j, e_s\> &= e_k(\G_{ij}^{\hphantom{i}s})-e_i(\G_{kj}^{\hphantom{k}s}) +
  \sum_{l=1}^{n}(\G_{ij}^{\hphantom{i}l}\G_{kl}^{\hphantom{k}s} - \G_{kj}^{\hphantom{k}l}\G_{il}^{\hphantom{i}s}
   - \G_{lj}^{\hphantom{l}s}(\G_{ki}^{\hphantom{k}l}-\G_{ik}^{\hphantom{i}l}))
  \\
  &= (\kappa + \la_i \la_k) (\delta_{ij}\delta_{ks} - \delta_{kj}\delta_{is}).
\end{split}
\end{equation}

For $a=1, \dots, q$, denote $L_a$ the eigendistribution of $S$ corresponding to the principal curvature $\mu_a$. If $\dim L_a (\, = p_a) > 1$, it easily follows from Codazzi equation that $L_a$ is integrable and its integral submanifolds are totally umbilical (or totally geodesic) in both $M^n$ and $\tM^{n+1}(\kappa)$, and the principal curvature $\mu_a$ is constant along these submanifolds.

\medskip

In the rest of the paper, we give the proof of Theorem~\ref{th:main}. We split it into several cases depending on the number $q$ of different principal curvatures. From Lemma~\ref{l:four} we know that $q \in \{1,2,3,4\}$. If $q=1$ on $\mU$, then $\mU$ is totally umbilical (or totally geodesic), and hence Einstein. The case $q=2$ is mostly known by~\cite[Lemma~3.12]{MV}. We consider it in Section~\ref{s:q=2} showing that a hypersurface $M^n \subset \tM^{n+1}(\kappa)$ satisfying the assumption of Theorem~\ref{th:main} is either a product hypersurface given in Example~\ref{ex:prod}, or a rotation hypersurface given in Example~\ref{ex:rotation}, as per Theorem~\ref{th:main}\eqref{it:thprod} and~\eqref{it:throt}, respectively. In Sections~\ref{s:q=3} and~\ref{s:q=4} we prove that there are no weakly Einstein non-Einstein hypersurfaces $M^n \subset \tM^{n+1}(\kappa)$ of dimension $n \ge 4$ having $q \in \{3,4\}$ principal curvatures (the fact that the case $n=q=4$ is not possible is established in~\cite[Lemma~3.13]{MV}). Both our proofs follow the same scheme: first we verify that the algebraic conditions~\eqref{eq:cRaa} are incompatible with the hypersurface being \emph{isoparametric} (that is, having constant principal curvatures). Next it will be easy to see that if all the multiplicities $p_a,\, a=1, \dots, q$ are greater than $1$, then the principal curvatures are forced to be constant by the algebraic conditions~\eqref{eq:cRaa} and Codazzi equations~\eqref{eq:codazzi}. Then we consider the possibilities of one or more of these multiplicities to be $1$ and show that the algebraic conditions~\eqref{eq:cRaa}, Codazzi equations~\eqref{eq:codazzi} and Gauss equations~\eqref{eq:gauss} again force the principal curvatures to be constant, which is a contradiction. Some parts of the proof rely on showing that certain large systems of polynomial equations have only constant solutions. In these parts, the proof is computer aided (but we explicitly give all the steps).

This leaves only the case $n=q=3$ (so that $p_1=p_2=p_3=1$). By~\cite[Lemma~5]{GHMV}, a $3$-dimensional non-Einstein Riemannian manifold is weakly Einstein if and only if its Ricci operator is of rank one. On the other hand, the algebraic conditions~\eqref{eq:cRaa} imply that the shape operator $S$ has rank $2$ and that the scalar curvature is constant. Hypersurfaces in spaces of constant curvature which satisfy these two conditions have been classified in~\cite{BFZ} (answering an open question on curvature homogeneous hypersurface of $\tM^{4}(\kappa)$ raised in \cite[p. 255]{BKV}). Neither of them has Ricci operator of rank one, by~\cite[Corollary~15]{BFZ}.

The proof is completed in Section~\ref{s:proof} with a local-to-global argument showing that $M^n$ cannot consist of several different pieces as in Examples~\ref{ex:prod} and~\ref{ex:rotation} potentially separated by umbilical points.


\section{Weakly Einstein hypersurfaces having 2 principal curvatures}
\label{s:q=2}

This case is effectively done in~\cite[Lemma~3.12]{MV}. We include a short proof for the sake of completeness, and also to explain the structure of the rotation hypersurfaces given in Example~\ref{ex:rotation}.

In the assumptions of Theorem~\ref{th:main}, suppose that the hypersurface $M^n \subset \tM^{n+1}(\kappa)$ has two principal curvatures everywhere and at no point the Ricci tensor is a multiple of the metric tensor.

Eliminating $\|R\|^2$ from~\eqref{eq:cRaa} we obtain
    \begin{equation} \label{eq:q=2}
      (p_1-1)(\mu_1^3+\mu_1^2 \mu_2 + 2 \kappa \mu_1) + (p_2-1)(\mu_2^3+\mu_2^2 \mu_1 + 2 \kappa \mu_2) = 0.
    \end{equation}
We consider two cases.

First suppose that $p_1, p_2 > 1$. From Codazzi equation~\eqref{eq:codazzi}, the principal curvature $\mu_1$ must be constant along the eigendistribution $L_1$. But then from~\eqref{eq:q=2}, the principal curvature $\mu_2$ also is. Similar argument applied to the eigendistribution $L_2$ shows that the hypersurface is isoparametric. By~\cite[Theorem~3.14, Theorem~3.29]{CR}, \cite[Section~4]{C1}, either $M^n$ is of constant curvature (which is not possible as $M^n$ is non-Einstein), or is a domain of the product $M^{p_1}(\kappa_1) \times M^{p_2}(\kappa_2)$ as given in Example~\ref{ex:prod}, where $\kappa_1^{-1} +\kappa_2^{-1} = \kappa^{-1}$ and $(p_1-1) \kappa_1^2 = (p_2-1) \kappa_2^2$ (the weakly Einstein condition) and, up to relabelling, $\kappa_1, \kappa_2 \kappa > 0$. Note that if $\kappa < 0$, one must have $p_1 < p_2$, and when $\kappa > 0$, one has $p_1 \ne p_2$, for otherwise $\kappa_2= \kappa_1$, and so $M^n$ is Einstein. In both cases, $n (\,= p_1+p_2) \ge 5$.

Now assume that $p_1 = 1$. Then $p_2=n-1 \ge 2$, and equation~\eqref{eq:q=2} gives $\mu_2^2 +\mu_1\mu_2 + 2\kappa=0$. By~\cite[Theorem~4.2]{dCD}, $M^n$ is a domain on a rotation hypersurface, as given in Example~\ref{ex:rotation}, whose principal curvatures are given in~\cite[Proposition~3.2]{dCD}. From this we find $\frac{d^2}{ds^2}(x_1^2)=2 \delta$, where $s$ is an arc length parameter, and $\delta = 1$ for the metrics $g_1$ in cases~\eqref{it:rotss} and~\eqref{it:roths} of Example~\ref{ex:rotation}, $\delta = -1$ for the metric $g_1$ in case~\eqref{it:rothh}, and $\delta = 0$ for the metric $g_1$ in case~\eqref{it:rothp}. It follows that $x_1^2=\delta s^2 + as + b$ for some $a, b \in \br$. When $\delta \ne 0$ (cases~\eqref{it:rotss}, \eqref{it:rothh} and~\eqref{it:roths}) we can shift $s$ to make $a=0$. Then the expression for $x_1$ and the inequalities given in cases \eqref{it:rotss}, \eqref{it:rothh} and~\eqref{it:roths} of Example~\ref{ex:rotation} follow from the conditions $g_1(\gamma,\gamma)= \kappa^{-1}$ and $g_1(\dot{\gamma},\dot{\gamma})= 1$ for the profile curve $\gamma(s)=(x_1(s), x_2(s), x_3(s))$ and the form of the corresponding metric $g_1$. In case~\eqref{it:rothp} we obtain $x_1^2=as + b$, and it is easy to see that if $a = 0$, our hypersurface has constant curvature. So shifting $s$ and changing its sign if necessary we can take $a > 0$ and $b=0$, as in Example~\ref{ex:rotation}\eqref{it:rothp}.

\section{Weakly Einstein hypersurface cannot have 3 principal curvatures}
\label{s:q=3}

Suppose there is a hypersurface $M^n \subset \tM^{n+1}(\kappa)$ satisfying the assumptions of Theorem~\ref{th:main} and having three principal curvatures everywhere. By the argument in the second last paragraph of Section~\ref{s:pre}, we can additionally assume throughout this section that $n \ge 4$.

We first prove the following.

\begin{lemma}\label{l:q3}
    In the above notation and assumptions, the following holds.
  \begin{enumerate}[label=\emph{(\alph*)},ref=\alph*]
    \item \label{it:f1f2}
    The principal curvatures $\mu_a,\, a=1,2,3$, and their multiplicities $p_a$ satisfy the equations $f_1=f_2=0$, where
    \begin{equation} \label{eq:f1f2}
    \begin{split} 
       f_1 & = (1-p_1)\mu_1^2+ (1-p_2)\mu_2^2 + (1-p_3)\mu_3^2 + \mu_1\mu_2 + \mu_2\mu_3 + \mu_3\mu_1  + 2\kappa,\\
       f_2 & = (p_1 - 1)\mu_1 (\mu_1^2 + 2\kappa) + \mu_1(p_2\mu_2^2 + \mu_3\mu_2 + p_3 \mu_3^2)+ 2\kappa (p_2\mu_2 + p_3\mu_3) + \mu_2^2\mu_3 + \mu_2\mu_3^2.
     \end{split}
    \end{equation}

    \item \label{it:noniso3}
    The hypersurface $M^n \subset \tM^{n+1}(\kappa)$ cannot be isoparametric.

    \item \label{it:no1103}
    If $p_1=p_2=1$, then the set of points of $M^n$ at which $\mu_3=0$ has empty interior.

    \item \label{it:1implies3}
    Suppose that for some $a \in \{1,2,3\}$, the function $\mu_a$ is constant on any integral submanifold of $L_a$ \emph{(}recall that by~\eqref{eq:codazzi}, the eigendistribution $L_a$ is integrable, and $\mu_a$ is automatically constant on any integral submanifold of $L_a$ when $p_a > 1$\emph{)}. Then all three principal curvatures $\mu_1,\mu_2$ and $\mu_3$ are also constant on any integral submanifold of $L_a$.

    \item \label{it:exists1}
    Up to relabelling, we have either $p_1=p_2=1 < p_3$ or $p_1 =1 < p_2,p_3$.
  \end{enumerate}
\end{lemma}
\begin{proof}
  For assertion~\eqref{it:f1f2}, let $F_b(\mu_1,\mu_2,\mu_3)$ be the polynomial obtained by subtracting the equation~\eqref{eq:cRaa} with $a=b_2$ from the equation~\eqref{eq:cRaa} with $a=b_1$ and dividing by $\mu_{b_1}-\mu_{b_2}$, where $\{b,b_1,b_2\}=\{1,2,3\}$. We obtain $F_1=p_1 \mu_1^2 (\mu_2+\mu_3) + 2\kappa p_1 \mu_1 + (p_2-1) \mu_2(\mu_2^2 + \mu_2 \mu_3 + 2 \kappa) + (p_3 - 1)\mu_3 (\mu_3^2 + \mu_2 \mu_3 + 2 \kappa)$, and $F_3$ which differs from $F_1$ by permuting subscripts $1$ and $3$ on the right-hand side. Subtracting $F_1$ from $F_3$ and dividing by $\mu_1-\mu_3$ we obtain $f_1$ as given in~\eqref{eq:f1f2}, and then we take $f_2=F_3+\mu_2 f_1$.

  To prove assertion~\eqref{it:noniso3}, we suppose that $M^n$ is isoparametric. Then by~\cite[Theorem~3.14]{CR}, \cite[Section~6]{C1}, one has $\kappa > 0$, and we can assume that $\tM^{n+1}(\kappa) = S^{n+1}(1)$. Furthermore, by~\cite[Theorem~3.26]{CR} and \cite[Section~2]{C2}, the principal curvatures and their multiplicities satisfy $p_1=p_2=p_3=p \in \{1,2,4,8\}$ and $\mu_1 = t, \, \mu_2=\frac{t-\sqrt{3}}{1+\sqrt{3}t}, \, \mu_3=\frac{t+\sqrt{3}}{1-\sqrt{3}t}$, where $t > \frac{1}{\sqrt{3}}$. Substituting into~\eqref{eq:f1f2} we obtain that $f_2 + t f_1$ may only be zero when $t =\sqrt{3}$, but substituting $t =\sqrt{3}$ into $f_1$ we get $5-6p$.   

  For assertion~\eqref{it:no1103}, let $p_1=p_2=1$ and assume that $\mu_3=0$ on an open domain $\mU \subset M^n$. Then from~\eqref{eq:f1f2} we obtain $\mu_1\mu_2=-2\kappa$ on $\mU$. By~\eqref{eq:gauss}, the components of the curvature tensor relative to the orthonormal basis $\{e_i\}$ are constant, and so the hypersurface $\mU$ is \emph{curvature homogeneous}. By \cite[Theorems~B,C]{Tsu}, a curvature homogeneous hypersurface in a space form $\tM^{n+1}(\kappa)$ (where $\kappa \ne 0$ and $n \ge 4$) either has constant curvature, or is isoparametric, or is congruent to a domain of a special hypersurface $M^4$ of the hyperbolic space $\tM^5(-1)$ (up to scaling). But by~\cite[Lemma~5.3(3)]{Tsu}, for such a hypersurface, $\mu_1 \mu_2 = 4$ contradicting the fact that $\mu_1\mu_2=-2\kappa=2$ (alternatively, one can deduce a contradiction from~\cite[Section~7]{BFZ}).

  For assertion~\eqref{it:1implies3}, suppose that $\mu_2$ is constant on any integral submanifold of the distribution $L_2$. Computing the resultant of $f_1$ and $f_2$ (which are given by~\eqref{eq:f1f2}) relative to $\mu_3$ we obtain a polynomial $\Phi$ of $\mu_1$ with coefficients depending on $\mu_2$ (and the constants $\kappa, p_1,p_2,p_3$) whose leading term is $(p_1 - 1)(p_1 + p_3 - 1)\mu_1^6$. If $p_1 > 1$, this is nonzero, which implies that $\mu_1$ is constant on any integral submanifold of the distribution $L_2$. If $p_1=1$, then the leading term of $\Phi$ relative to $\mu_1$ is $p_3(p_2 + p_3 - 1)\mu_2^2 \mu_1^4$. If $\mu_2 \ne 0$ at a point $x \in M^n$, then the restriction of $\Phi$ to the integral submanifold of $L_2$ passing through $x$ is a nonzero polynomial of $\mu_1$ with constant coefficients, which implies that $\mu_1$ is constant on that submanifold. If $\mu_2 = 0$ at a point $x \in M^n$, then substituting $p_1=1$ and $\mu_2=0$ into $f_1$ and $f_2$ we obtain $p_3=1$ (and $\mu_1\mu_3=-2\kappa$). By assertion~\eqref{it:no1103}, the set of such points $x \in M^n$ has an empty interior, which implies that $\mu_1$ is constant on all integral submanifolds of the distribution $L_2$. By a similar argument, $\mu_3$ is also constant on any such integral submanifold.

  Assertion~\eqref{it:exists1} follows from assertions~\eqref{it:1implies3} and~\eqref{it:noniso3} (and the fact that $n \ge 4$).
\end{proof}

We now separately consider two cases in Lemma~\ref{l:q3}\eqref{it:exists1}. In both cases, we show that equations~\eqref{eq:f1f2} together with Codazzi and Gauss equations will force all three principal curvatures to be constant, in contradiction with Lemma~\ref{l:q3}\eqref{it:noniso3}. Some of calculations in these cases are computer aided.

First consider the case when $p_1=1 < p_2, p_3$. Choose the labelling in such a way that $\la_1=\mu_1$. By Lemma~\ref{l:q3}\eqref{it:1implies3}, we have $e_i(\mu_a)=0$, for all $a=1,2,3$ and all $i > 1$. Note that $\mu_1$ cannot be locally constant, for otherwise by Lemma~\ref{l:q3}\eqref{it:1implies3}, $\mu_2$ and $\mu_3$ are also locally constant contradicting Lemma~\ref{l:q3}\eqref{it:noniso3}. Replacing $M^n$ by a smaller domain if necessary, we can assume that $\nabla \mu_1 \ne 0$ everywhere. It follows that the distribution $L_2 \oplus L_3$ is integrable, as is tangent to the level hypersurfaces of the function $\mu_1$, and so $\G_{ij}^{\hphantom{i}1}=\G_{ji}^{\hphantom{j}1}$ for $i,j > 1$. Then~\eqref{eq:codazzi} give $\G_{ij}^{\hphantom{i}k}=0$ when $\la_i=\la_j=\mu_a, \, \la_k = \mu_b$ with $b \ne a,1$ (in particular, $e_1$ is a geodesic vector field), $\G_{ij}^{\hphantom{i}1}=0$ when $i \ne j$ and $\G_{1i}^{\hphantom{1}j}=0$ when $\la_i \ne \la_j$. The latter means that $\nabla_1 L_a \subset L_a$ for $a =2,3$. We can then specify vector fields $e_i \in L_a,\, a=2,3$, by choosing them arbitrarily on a level hypersurface of the function $\mu_1$ and then extending parallelly along the (geodesic) integral curves of $e_1$. Then $\G_{1i}^{\hphantom{1}j}=0$, for all $i, j =1, \dots, n$. We also obtain $\G_{ii}^{\hphantom{i}1}=\G_{jj}^{\hphantom{j}1}$ when $\la_i=\la_j=\mu_a,\, a>1$, and so we denote $\G_{ii}^{\hphantom{i}1}=\phi_a$, where $\la_i=\mu_a,\, a=2,3$. Then the only remaining Codazzi equations are
\begin{equation}\label{eq:codazzi31}
  e_1(\mu_2)=\phi_2(\mu_2-\mu_1), \qquad e_1(\mu_3)=\phi_3(\mu_3-\mu_1).
\end{equation}
Computing $R_{1ii1},\, i > 1$, and $R_{jiij},\, \la_i =\mu_2,\, \la_j = \mu_3$, from Gauss equations~\eqref{eq:gauss} we find
\begin{equation}\label{eq:gauss31}
  \mu_2\mu_3+\phi_2\phi_3+\kappa=0,\quad e_1(\phi_2)=\phi_2^2+\mu_2\mu_1+\kappa, \quad e_1(\phi_3)=\phi_3^2+\mu_3\mu_1+\kappa.
\end{equation}
Furthermore, taking $p_1=1$ in \eqref{eq:f1f2} (and checking that $\mu_2+\mu_3$ cannot be zero on a subset of $M^n$ with nonempty interior, for otherwise all three principal curvatures $\mu_a$ are constant in contradiction with Lemma~\ref{l:q3}\eqref{it:noniso3}) we can solve $f_1$ for $\mu_1$ and substitute the resulting expression into~\eqref{eq:codazzi31},~\eqref{eq:gauss31} and into $f_2=0$. The latter gives the equation $f_3=0$, where
    \begin{equation*} 
       f_3 = p_2(p_2 - 1)\mu_2^4 + p_3(p_3 - 1)\mu_3^4 + (2p_2p_3 - p_2 - p_3 + 1)\mu_2^2\mu_3^2 + 2\kappa(p_2 + p_3 - 1)\mu_2\mu_3
    \end{equation*}
Now we differentiate $f_3$ along $e_1$ and substitute $e_1(\mu_2)$ and $e_1(\mu_3)$ from~\eqref{eq:codazzi31} (with $\mu_1$ having already been expressed in terms of $\mu_2$ and $\mu_3$ from $f_1$). This gives a polynomial equation $f_4=0$, where $f_4$ is linear in $\phi_2, \phi_3$, with coefficients polynomial in $\mu_2,\mu_3$. Differentiating $f_4$ along $e_1$ and substituting $e_i(\mu_a)$ from ~\eqref{eq:codazzi31} and $e_i(\phi_a)$ from~\eqref{eq:gauss31} we obtain an equation $f_5=0$, where $f_5$ is quadratic in $\phi_2, \phi_3$ with coefficients polynomial in $\mu_2,\mu_3$. Eliminating $\phi_2, \phi_3$ from $f_4=f_5=0$ and the first equation of~\eqref{eq:gauss31} by calculating the corresponding resultants we get a polynomial equation $f_6=0$ for $\mu_2,\mu_3$ with constant coefficients, where $f_6=g_1^4g_2^2g_3^2g_4^2$, and $g_1$ and $g_2$ being quadratic in $\mu_2,\mu_3$, $g_3$ being cubic, and $g_4$ is of degree $14$. Computing the resultants of each of $g_1,g_2,g_3$ and $g_4$ with $f_3$ relative to $\mu_3$ we obtain four polynomials in $\mu_2$ with constant coefficients whose leading terms are $4 (p_2 - 1) (p_2 + p_3 - 1)^{3} \mu_2^{8}, 4 (p_3 - 1) (p_2 + p_3 - 1)^{3} \mu_2^{8},  p_2 (p_2 - 1) p_3^2 (p_3 - 1)^2 (p_2 + p_3 - 1)^{4} \mu_2^{12}$ and $p_2^5 (p_2 - 1)^6 p_3^3 (p_3 - 1)^4 (p_2 + p_3 - 1)^{18} \mu_2^{56}$ respectively. It follows that $\mu_2$ is a constant on $M^n$. By a similar argument, $\mu_3$ is also a constant, as so is $\mu_1$ by $f_1=0$. It follows that $M^n$ is isoparametric, in contradiction with Lemma~\ref{l:q3}\eqref{it:no1103}. 

Now consider the case when $p_1=p_2=1 < p_3$. Choose the labelling in such a way that $\la_1=\mu_1, \, \la_2=\mu_2$, and $\la_i= \mu_3$ for $i > 2$. From Lemma~\ref{l:q3}\eqref{it:no1103} we can assume that $\mu_3 \ne 0$ anywhere on $M^n$. Substituting $p_1=p_2=1$ into $f_1$ and $f_2-\mu_2 f_1$ given by~\eqref{eq:f1f2} we obtain
\begin{equation}\label{eq:11p3}
  (\mu_1+\mu_2)\mu_3=-2\kappa, \qquad \mu_1\mu_2=(n-3)\mu_3^2.
\end{equation}
From equations~\eqref{eq:11p3} it is easy to see that neither of $\mu_a$ can be locally constant, for otherwise the other two also are, in contradiction with Lemma~\ref{l:q3}\eqref{it:noniso3}.

By Lemma~\ref{l:q3}\eqref{it:1implies3}, we have $e_i(\mu_a)=0$, for all $a=1,2,3$ and all $i > 2$. Moreover, the distribution $L_3$ is integrable, and its integral submanifolds are totally umbilical (or totally geodesic) in both $M^n$ and $\tM^{n+1}(\kappa)$. Codazzi equations~\eqref{eq:codazzi} give $\G_{11}^{\hphantom{1}i}=\G_{22}^{\hphantom{2}i}=0$ for $i > 2$, and $\G_{ij}^{\hphantom{i}1}= \G_{ij}^{\hphantom{i}2}= 0$ for all $i \ne j,\,  i, j > 2$. We also obtain $\G_{ii}^{\hphantom{i}a}=\G_{jj}^{\hphantom{j}a}$ for $i, j > 2, \, a=1,2$, and so we denote $\G_{ii}^{\hphantom{i}a}=\phi_a$, where $i> 2 \ge a$. The remaining Codazzi equations are
\begin{gather}\label{eq:codazzi32}
    e_1(\mu_3)=\phi_1(\mu_3-\mu_1), \qquad e_2(\mu_3)=\phi_2(\mu_3-\mu_2), \\
    e_2(\mu_1)=\gamma_2(\mu_1-\mu_2), \qquad e_1(\mu_2)=\gamma_1(\mu_2-\mu_1), \label{eq:codazzi3212} \\
    \G_{i1}^{\hphantom{i}2}(\mu_2-\mu_1) = \G_{21}^{\hphantom{2}i}(\mu_3-\mu_1) = \G_{12}^{\hphantom{1}i}(\mu_3-\mu_2), \label{eq:codazzi32g}
\end{gather}
where $i>2$, and where we denote $\gamma_1=\G_{22}^{\hphantom{2}1},\, \gamma_2=\G_{11}^{\hphantom{1}2}$. For $k \ne s,\, k,s > 2$, we compute from~\eqref{eq:gauss} and \eqref{eq:codazzi32g}: $0=R_{k11s} = 2\G_{k1}^{\hphantom{k}2} \G_{s1}^{\hphantom{s}2} (\mu_3-\mu_2)^{-1} (\mu_1-\mu_2)$. It follows that $\G_{k1}^{\hphantom{k}2}=0$, for all $k > 2$, and so $\G_{21}^{\hphantom{2}k}=\G_{12}^{\hphantom{2}k}=0$, for all $k > 2$, by~\eqref{eq:codazzi32g}. Therefore the distribution $L_1 \oplus L_2$ is integrable and its integral submanifolds are totally geodesic in $M^n$. We can now specify the vector fields $e_i,\, i > 2$, choosing them arbitrarily on an integral submanifold of the distribution $L_3$, and then parallely translating along totally geodesic integral submanifolds of the distribution $L_1 \oplus L_2$. Then we obtain $\G_{1i}^{\hphantom{1}j}=\G_{2i}^{\hphantom{2}j}=0$, for all $i, j > 2$.

Computing $R_{1221}$ and $R_{kijk}, \, k>2, \, i,j \in \{1,2\}$, from Gauss equation~\eqref{eq:gauss} we obtain
\begin{gather} \label{eq:r1221}
  e_1(\gamma_1)+e_2(\gamma_2)-\gamma_1^2-\gamma_2^2=\kappa+\mu_1\mu_2,\\
  e_i(\phi_j)=\phi_i(\phi_j-\gamma_j), \quad e_i(\phi_i)=\phi_j\gamma_j+\phi_i^2+\mu_i\mu_3+\kappa, \quad \{i,j\}=\{1,2\}. \label{eq:phiij}
\end{gather}
From~\eqref{eq:11p3} we obtain $\mu_3 =-2\kappa(\mu_1+\mu_2)^{-1}$ and $f(\mu_1,\mu_2):=(\mu_1+\mu_2)^2\mu_1\mu_2-4\kappa^2(n-3)=0$. Differentiating $f$ along $e_1$ and $e_2$ and using~\eqref{eq:codazzi3212} we obtain $e_i(\mu_i) = \gamma_i\mu_i(\mu_i-\mu_j) (\mu_i+3\mu_j) (\mu_j(3\mu_i+\mu_j))^{-1}$ for $\{i,j\}=\{1,2\}$, which, together with \eqref{eq:codazzi3212}, gives us all the derivatives $e_i(\mu_j), \, i,j \in \{1,2\}$, in terms of $\mu_1,\mu_2,\gamma_1$ and $\gamma_2$. Substituting $\mu_3=-2\kappa(\mu_1+\mu_2)^{-1}$ into~\eqref{eq:codazzi32} we obtain $\phi_i = -2\kappa\gamma_i(\mu_i-\mu_j)^2 (\mu_j(\mu_i^2+\mu_i\mu_j+2\kappa)(3\mu_i+\mu_j))^{-1}$ for $\{i,j\}=\{1,2\}$. Then from~\eqref{eq:phiij} we can express $e_i(\gamma_j),  \, i,j \in \{1,2\}$ in terms of $\mu_1,\mu_2,\gamma_1$ and $\gamma_2$. Now substituting these expressions to~\eqref{eq:r1221} we obtain a polynomial equation $F=h_{1}\gamma_1^2 + h_{2}\gamma_2^2 + h_{3}=0$, where $h_1, h_2$ and $h_3$ are polynomials in $\mu_1$ and $\mu_2$. Differentiating $F$ along $e_i, \, i=1,2$ and using the known expressions for the derivatives of $\gamma_1$ and $\gamma_2$, we obtain equations $F_i=h_{i1}\gamma_1^2 + h_{i2}\gamma_2^2 + h_{i3}=0$, where $h_{i1}, h_{i2}$ and $h_{i3}$ are again polynomials in $\mu_1, \mu_2$. The equations $F=F_1=F_2=0$ can be viewed as a system of linear equations for the vector $(\gamma_1^2, \gamma_2^2, 1)^t$ with coefficients polynomial in $\mu_1, \mu_2$. Computing its determinant we get an equation $F_3=0$, where $F_3$ is a symmetric polynomial in $\mu_1, \mu_2$ of the form $32 \mu_1^2\mu_2^2(\mu_1 + \mu_2)^2(\mu_1 - \mu_2)^6(\mu_1 + 3\mu_2)^6(3\mu_1 + \mu_2)^6 (\mu_2^2+\mu_1\mu_2+ 2\kappa)^2(\mu_1^2 + \mu_1\mu_2 + 2\kappa)^2 F_4(\mu_1,\mu_2)$, where $F_4$ is a symmetric polynomial of $\mu_1, \mu_2$. It is easy to see from~\eqref{eq:11p3} that neither of the above terms (except possibly $F_4$) can be zero on an open subset of $M^n$, as otherwise all three $\mu_a$ are locally constant (and this also guarantees that no denominator in the course of our calculation can be zero). Expressing $F_4$ in terms of the elementary symmetric functions $\sigma_1=\mu_1+\mu_2$ and $\sigma_2=\mu_1\mu_2$ and substituting $\sigma_2=4\kappa^2(n - 3) \sigma_1^{-2}$ from~\eqref{eq:11p3} we obtain a polynomial in $\sigma_1$ with constant coefficients whose leading term is
$192 \kappa^6 (n-1)(n-2) \sigma_1^{14}$. It follows that $\sigma_1$ is a constant which easily implies by~\eqref{eq:11p3} that all three $\mu_a$ are constant on $M^n$, in contradiction with Lemma~\ref{l:q3}\eqref{it:noniso3}.

It follows that in the assumptions of Theorem~\ref{th:main}, the set of points at which $M^n \subset \tM^{n+1}(\kappa)$ has three principal curvatures has empty interior.

\section{Weakly Einstein hypersurface cannot have 4 principal curvatures}
\label{s:q=4}

Suppose there is a hypersurface $M^n \subset \tM^{n+1}(\kappa)$ satisfying the assumptions of Theorem~\ref{th:main} and having four principal curvatures everywhere. By~\cite[Lemma~3.13]{MV} we can assume that not all four of them have multiplicity $1$, or equivalently, that $n \ge 5$.

From~\eqref{eq:la4} we obtain the equations $f_1=f_2=f_3=0$, where
  \begin{equation}\label{eq:vnew}
    f_1=\sum_{a=1}^{4}\mu_a,\quad
    f_2=\sum_{a=1}^{4}(2p_a-1)\mu_a^2 -4\kappa,\quad
    f_3=\sum_{a<b<c} \mu_a\mu_b\mu_c - 2\kappa \sum_{a=1}^{4}p_a\mu_a.
  \end{equation}
Note that from $f_2=0$ we see that $\kappa > 0$. We start somewhat similarly to the case $q=3$ considered in Section~\ref{s:q=3}.

\begin{lemma} \label{l:q4}
   In the above notation and assumptions, the following holds.
  \begin{enumerate}[label=\emph{(\alph*)},ref=\alph*]
    \item \label{it:noniso4}
    The hypersurface $M^n \subset \tM^{n+1}(\kappa)$ cannot be isoparametric.

    \item \label{it:1implies4}
    Suppose that for some $a \in \{1,2,3,4\}$, the function $\mu_a$ is constant on any integral submanifold of $L_a$ \emph{(}this is always true when $p_a > 1$\emph{)}. Then all four principal curvatures $\mu_1,\mu_2,\mu_3$ and $\mu_4$ are also constant on any integral submanifold of $L_a$.

    \item \label{it:exists14}
    At least one of the multiplicities $p_a,\, a=1,2,3,4$, equals $1$ \emph{(}and at least one is greater than $1$, as $n \ge 5$\emph{)}.
  \end{enumerate}
\end{lemma}
\begin{proof}
  To prove assertion~\eqref{it:noniso4}, we suppose that $M^n$ is isoparametric. Since $\kappa>0$, we can take $\kappa=1$. By~\cite[Theorem~3.26]{CR} and \cite[Satz~1]{Mun}, the principal curvatures and their multiplicities satisfy $p_1=p_3,\; p_2=p_4$ and $\mu_1 = t, \, \mu_2=\frac{t-1}{t+1}, \, \mu_3=-\frac{1}{t},\, \mu_4=\frac{t+1}{1-t}$, where $t > 1$. Substituting into~\eqref{eq:vnew}, from $f_1=0$ we find $t=\sqrt{2}+1$, and then $f_2=0$ gives $3(p_1+p_2)=4$, a contradiction. 

  For assertion~\eqref{it:1implies4}, suppose $a=1$. Substituting $\mu_4=-\mu_1-\mu_2-\mu_3$ into $f_2$ and $f_3$ and computing the resultant of the two resulting polynomials relative to $\mu_3$ we obtain a polynomial equation for $\mu_2$ with constant coefficients whose leading term is
  $4 (p_2 + p_4 - 1)(p_2 + p_3 - 1)\mu_2^6$. It follows that $\mu_2$ is constant on the integral submanifold of the distribution $L_1$. By a similar argument, $\mu_3$ and $\mu_4$ are also constant.

  Assertion~\eqref{it:exists14} follows from assertions~\eqref{it:noniso4} and~\eqref{it:1implies4}.
\end{proof}

By Lemma~\ref{l:q4}\eqref{it:exists14}, out of the four multiplicities $p_a$, we can have either one, or two, or three ones. We consider these three cases below. In all three cases, we will show that the equations~\eqref{eq:vnew} together with Codazzi and Gauss equations will imply that all the principal curvatures are constant, which contradicts Lemma~\ref{l:q4}\eqref{it:noniso4}. In several places, our calculations are computer aided.

\subsection{\texorpdfstring{Case 1. $p_1=1<p_2,p_3,p_4$}{p{\textoneinferior} = 1 < p{\texttwoinferior}, p{\textthreeinferior}, p{\textfourinferior}}}
\label{ss:1p1}
Up to relabelling we can assume that $\mu_1=\la_1$. From Lemma~\ref{l:q4}\eqref{it:1implies4} we obtain $e_i(\mu_a)=0$, for all $a=1,2,3,4$ and all $i > 1$. Note that $\mu_1$ cannot be locally constant, for otherwise by Lemma~\ref{l:q4}\eqref{it:1implies4}, $\mu_2, \mu_3$ and $\mu_4$ are also locally constant contradicting Lemma~\ref{l:q4}\eqref{it:noniso4}. It follows that the distribution $L_2 \oplus L_3 \oplus L_4$ is integrable, as it is tangent to the level hypersurfaces of the function $\mu_1$, and so $\G_{ij}^{\hphantom{i}1} = \G_{ji}^{\hphantom{j}1}$ for $i,j > 1$. Then~\eqref{eq:codazzi} gives $\G_{ij}^{\hphantom{i}k}=0$ when $\la_i=\la_j=\mu_a, \, \la_k = \mu_b$ with $b \ne a,1$ (in particular, $e_1$ is a geodesic vector field), $\G_{ij}^{\hphantom{i}1}=0$ when $i \ne j$ and $\G_{1i}^{\hphantom{1}j}=0$ when $\la_i \ne \la_j$. The latter means that $\nabla_1 L_a \subset L_a$ for $a =2,3,4$. We can then specify vector fields $e_i \in L_a,\, a=2,3,4$, by choosing them arbitrarily on a level hypersurface of the function $\mu_1$ and then extending parallely along the (geodesic) integral curves of $e_1$. Then $\G_{1i}^{\hphantom{1}j}=0$, for all $i, j =1, \dots, n$. We also obtain $\G_{ii}^{\hphantom{i}1}=\G_{jj}^{\hphantom{j}1}$ when $\la_i=\la_j=\mu_a,\, a>1$, and so we denote $\G_{ii}^{\hphantom{i}1}=\phi_a$, where $\la_i=\mu_a,\, a=2,3,4$. Then the only remaining Codazzi equations are
\begin{equation}\label{eq:codazzi41}
  e_1(\mu_a)=\phi_a(\mu_a-\mu_1), \qquad a=2,3,4,
\end{equation}
and
\begin{equation}\label{eq:codazzi41abc}
\G_{ij}^{\hphantom{i}k}(\mu_c-\mu_b)=\G_{jk}^{\hphantom{j}i}(\mu_a-\mu_c)=\G_{ki}^{\hphantom{k}j}(\mu_b-\mu_a) :=\psi_{ijk},
\end{equation}
where $\la_i=\mu_a,\, \la_j=\mu_b,\, \la_k=\mu_c$ and $\{a,b,c\}=\{2,3,4\}$. Note that $\psi_{ijk}$ is symmetric by all three indices.

Computing $R_{1ii1},\, i > 1$, from Gauss equations~\eqref{eq:gauss} we find
\begin{equation}\label{eq:gauss41}
  e_1(\phi_a)=\phi_a^2+\mu_a\mu_1+\kappa, \quad a=2,3,4.
\end{equation}

Now let $\la_i =\mu_a,\, \la_j = \mu_b$, where $a \ne b, \, a,b > 1$. Gauss equations~\eqref{eq:gauss} for $R_{jiij}$ give $\mu_a\mu_b+\phi_a\phi_b+\kappa=2 (\mu_c-\mu_a)^{-1}(\mu_c-\mu_b)^{-1} \sum_{k:\la_k=\mu_c} \psi_{ijk}^2$, where $\{a,b,c\}=\{2,3,4\}$. It follows that $\sum_{k:\la_k=\mu_c} \psi_{ijk}^2$ does not depend on the choice of the unit vectors $e_i \in L_a$ and $e_j \in L_b$. Similarly, $\sum_{j:\la_j=\mu_b} \psi_{ijk}^2$ does not depend on the choice of the unit vectors $e_i \in L_a$ and $e_k \in L_c$ (where  $\{a,b,c\}=\{2,3,4\}$). Suppose that locally the functions $\psi_{ijk}$ are not all identically zero. Then for a fixed $i$ such that $e_i \in L_a$, the $p_b \times p_c$ matrix $\psi_{ijk}$, where $\la_j=\mu_b,\, \la_k =\mu_c$, is a nonzero multiple of an orthogonal matrix. In particular, we have $p_b=p_c$, and by a similar argument, $p_a=p_c$, so that $p_2=p_3=p_4=:m$. We denote $\Psi = 2m^{-2} \sum_{i:\la_i=\mu_a, \, j:\la_j=\mu_b, \, k:\la_k=\mu_c} \psi_{ijk}^2$, if locally the functions $\psi_{ijk}$ are not all identically zero, and $\Psi = 0$ otherwise. Then for $\{a,b,c\}=\{2,3,4\}$ we obtain
\begin{equation}\label{eq:gauss41ab}
  \mu_a\mu_b+\phi_a\phi_b+\kappa=(\mu_c-\mu_a)^{-1}(\mu_c-\mu_b)^{-1}\Psi,\quad \text{and either } \Psi=0, \text{ or } p_2=p_3=p_4 (\, =m).
\end{equation}
We first consider the case when $\Psi=0$ (locally). Then from~\eqref{eq:gauss41ab} we get $(\mu_2\mu_3+\kappa) (\mu_3\mu_4+\kappa) (\mu_4\mu_2+\kappa) = -\rho^2$ for $\rho:=\phi_2\phi_3\phi_4 \ne 0$ (if $\rho = 0$, then by~\eqref{eq:codazzi41} and Lemma~\ref{l:q4}\eqref{it:1implies4}, we get a contradiction with Lemma~\ref{l:q4}\eqref{it:noniso4}), and then $\phi_a=-\rho (\mu_b\mu_c + \kappa)^{-1}$, where $\{a,b,c\}=\{2,3,4\}$. Substituting $\mu_1=-\mu_2-\mu_3-\mu_4$ into $f_2, f_3$ given by equations~\eqref{eq:vnew} we obtain two polynomial equations for $\mu_2, \mu_3, \mu_4$. Differentiating them along $e_1$ and substituting~\eqref{eq:codazzi41} and the above expressions for $\phi_a$ we obtain another two polynomial equations for $\mu_2, \mu_3, \mu_4$. Eliminating first $\mu_4$ and then $\mu_3$ we obtain a polynomial $P$ of degree $72$ in $\mu_2$ whose constant term is the product of $\kappa^{36}$ and a certain polynomial $P'$ in $p_3$ and $p_4$ with constant coefficients. Substituting $p_3=2+t_3,\, p_4=2+t_4$ we find that $P'$ has all its coefficients positive and a nonzero constant term. It follows that the constant term of $P$ is not zero, and so $P$ is a nontrivial polynomial of $\mu_2$. But then $\mu_2$ is locally constant, and by a similar argument, $\mu_3$ and $\mu_4$ (and then also $\mu_1$) are, in contradiction with Lemma~\ref{l:q4}\eqref{it:noniso4}.

Now suppose $\Psi$ is locally nonzero. From~\eqref{eq:gauss41ab} we obtain $\sum_{a=2}^4 \eta_a=\sum_{a=2}^4 \mu_a\eta_a=0$, where $\eta_a = \mu_b\mu_c + \phi_b\phi_c + \kappa$, with $\{a,b,c\}=\{2,3,4\}$. Differentiating the equation $\sum_{a=2}^4 \eta_a=0$ along $e_1$ and substituting~\eqref{eq:codazzi41},\eqref{eq:gauss41}, and then using the equations $\phi_b\left(\sum_{a=2}^{4}\eta_a\right)=0$ for $b=2,3,4$ and $\mu_1=-\mu_2-\mu_3-\mu_4$ from~\eqref{eq:vnew} we obtain $\sum_{a=2}^4 \phi_a\eta_a=0$. As $\eta_a \ne 0$, it follows that $\phi_a =\alpha \mu_a +\beta,\, a=2,3,4$, for some functions $\alpha$ and $\beta$. Substituting these expressions to~\eqref{eq:gauss41ab} and eliminating $\Psi$ we find $(\alpha^2+1) \sigma_2+2 \alpha \beta \sigma_1+3 (\beta^2+\kappa) = (\alpha^2+1)(9 \sigma_3 -\sigma_1 \sigma_2) + 2\alpha \beta (3\sigma_2 - \sigma_1^2)=0$, where $\sigma_1=\mu_2+\mu_3+\mu_4,\, \sigma_2=\mu_2\mu_3+\mu_3\mu_4+\mu_4\mu_2$ and $\sigma_3=\mu_2\mu_3\mu_4$. From~\eqref{eq:vnew} we obtain $\sigma_3-\sigma_2 \sigma_1-2 \kappa (m -1) \sigma_1=m \sigma_1^2-(2 m-1) \sigma_2-2 \kappa=0$.
Differentiating the former equation along $e_1$ and using~\eqref{eq:codazzi41},\eqref{eq:gauss41} and the fact that $\phi_a =\alpha \mu_a +\beta,\, a=2,3,4$ we get $4 m \alpha \sigma_1^3-(10 m-3) \alpha \sigma_1 \sigma_2+2 \beta(2 m  + 1) \sigma_1^2+3 \alpha(2 m - 1) \sigma_3-2 \beta(2 m -1) \sigma_2=0$. Now we can express $\sigma_2$ and $\sigma_3$ in terms of $\sigma_1$ from the second and the third last equations and substitute to the remaining three equations. Eliminating $\beta$ and then $\alpha$ we get a polynomial equation for $\sigma_1$ with constant coefficients whose leading term equals $576 m^4 (m -1)^2 (m -3)^2 (3 m -1)^2 (2 m -1)^2 (m +1)^2 \sigma_1^{22}$. As $m \ge 2$, this is nonzero unless $m=3$. But when $m=3$ we get a polynomial equation for $\sigma_1$ with the leading term $\kappa^2 \sigma_1^{18}$. It follows that $\sigma_1$ is locally constant, and then $\sigma_2$ and $\sigma_3$ also are. But then all the $\mu_a$ are constant which contradicts Lemma~\ref{l:q4}\eqref{it:noniso4}.

\subsection{\texorpdfstring{Case 2. $p_1=p_2=1<p_3,p_4$}{p{\textoneinferior} = p{\texttwoinferior} = 1 < p{\textthreeinferior}, p{\textfourinferior}}}
\label{ss:2p1}
Choose the labelling in such a way that $\la_1=\mu_1, \, \la_2=\mu_2$. Throughout this subsection, we use the following notation convention: $a, b, c \in \{1,2\}$, and then $i,j,k$ are such that $\la_i=\la_j=\la_k=\mu_3$, and $r,s,t$ are such that $\la_r=\la_s=\la_t=\mu_4$.

We have three algebraic equations~\eqref{eq:vnew} for $\mu_1, \mu_2, \mu_3$ and $\mu_4$ which in this case take the form
  \begin{equation}\label{eq:vnew2}
  \begin{split}
    f_1&=\mu_1+\mu_2+\mu_3+\mu_4=0,\\
    f_2&=\mu_1^2 + \mu_2^2 +(2p_3-1)\mu_3^2 + (2p_4-1)\mu_4^2 -4\kappa=0,\\
    f_3&=(\mu_1+\mu_2)\mu_3\mu_4+ \mu_1\mu_2(\mu_3+\mu_4) - 2\kappa (\mu_1+\mu_2+p_3\mu_3+p_4\mu_4)=0.
  \end{split}
  \end{equation}
These equations are symmetric relative to $\mu_1$ and $\mu_2$. If $p_3 = p_4$, we obtain $\mu_2=-\mu_1, \; \mu_4=-\mu_3$ and $\mu_1^2 + (2p_3-1)\mu_3^2 = 2 \kappa$. If $p_3 \ne p_4$, we can express $\mu_3$ and $\mu_4$ in terms of $\sigma_1=\mu_1+\mu_2$ and $\sigma_2=\mu_1\mu_2$, and we also obtain a nontrivial polynomial relation for $\sigma_1$ and $\sigma_2$. In both cases, if at least one of $\mu_1,\mu_2,\mu_3$ or $\mu_4$ is locally constant, then so are the other three.

By Lemma~\ref{l:q4}\eqref{it:1implies4}, we have $e_i(\mu_m)=e_r(\mu_m)=0$, for all $m=1,2,3,4$. Moreover, the distributions $L_3$ and $L_4$ are integrable, and their integral submanifolds are totally umbilical (or totally geodesic) in both $M^n$ and $\tM^{n+1}(\kappa)$. Codazzi equations~\eqref{eq:codazzi} give $\G_{ij}^{\hphantom{1}r}=\G_{ij}^{\hphantom{2}a}=\G_{rs}^{\hphantom{1}i}=\G_{rs}^{\hphantom{2}a}=0$ and $\G_{aa}^{\hphantom{i}i}= \G_{aa}^{\hphantom{i}r}= 0$. We also obtain $\G_{ii}^{\hphantom{i}a}=\G_{jj}^{\hphantom{j}a}$ for all $i, j$ and $a$; denote $\G_{ii}^{\hphantom{i}a}=\phi_3^a$. Similarly, $\G_{rr}^{\hphantom{r}a}=\G_{ss}^{\hphantom{s}a}$ for all $r, s$ and $a$; denote $\G_{rr}^{\hphantom{r}a}=\phi_4^a$. We also denote
\begin{equation*}
\G_{aa}^{\hphantom{a}b}= :\gamma_a, \quad \eta_i:=\G_{i2}^{\hphantom{i}1} (\mu_1-\mu_2), \quad \xi_r:=\G_{r2}^{\hphantom{i}1} (\mu_1-\mu_2), \quad \Phi_{ai}^{\hphantom{a}r} := \Gamma_{ai}^{\hphantom{a}r} (\mu_4-\mu_3),
\end{equation*}
for $a \ne b$. In these notation, the remaining Codazzi equations~\eqref{eq:codazzi} take the following form.
\begin{equation}\label{eq:codazzi2n}
\begin{gathered}
e_a(\mu_b)=\gamma_b(\mu_b-\mu_a), \quad \G_{ab}^{\hphantom{a}i}(\mu_3-\mu_b) = \eta_i, \quad \G_{ab}^{\hphantom{a}r}(\mu_4-\mu_b) = \xi_r \qquad a \ne b,\\
e_a(\mu_3)=\phi_3^a(\mu_3-\mu_a), \quad e_a(\mu_4)=\phi_4^a(\mu_4-\mu_a), \quad  \G_{ir}^{\hphantom{i}a}(\mu_a-\mu_4) = \G_{ri}^{\hphantom{r}a}(\mu_a-\mu_3) = \Phi_{ai}^{\hphantom{a}r}.
\end{gathered}
\end{equation}
Using~\eqref{eq:codazzi2n}, we can write down Gauss equations~\eqref{eq:gauss} for $R_{aija}$ and $R_{arsa}$ as the follows (where $\{a,b\}=\{1,2\}$ and the summations are over all $r$ such that $\la_r=\mu_4$ and over all $i$ such that $\la_i=\mu_3$, respectively): 
\begin{align}
  &
  \begin{aligned}
  2 (\mu_4-\mu_3)^{-1}(\mu_a-\mu_4)^{-1} \sum_r \Phi_{aj}^{\hphantom{a}r} \Phi_{ai}^{\hphantom{a}r} &= (e_a(\phi_3^a) - \phi_3^b \gamma_a - (\phi_3^a)^2 - \mu_a\mu_3 - \kappa) \delta_{ij} \\
  &+ 2 (\mu_3-\mu_b)^{-1}(\mu_a-\mu_b)^{-1} \eta_i \eta_j,
  \end{aligned} \label{eq:Gauss2aija}
  \\
  &
  \begin{aligned}
  2 (\mu_3-\mu_4)^{-1}(\mu_a-\mu_3)^{-1} \sum_i \Phi_{ai}^{\hphantom{a}s} \Phi_{ai}^{\hphantom{a}r} &= (e_a(\phi_4^a) - \phi_4^b \gamma_a - (\phi_4^a)^2 - \mu_a\mu_4 - \kappa)\delta_{rs}\\
  &+ 2 (\mu_4-\mu_b)^{-1}(\mu_a-\mu_b)^{-1} \xi_r \xi_s.
  \end{aligned} \label{eq:Gauss2arsa}
\end{align}
Gauss equations~\eqref{eq:gauss} for $R_{ikis}, R_{rsrk}$ and $R_{rijs}$ give:
\begin{gather}
  \sum_{c=1}^{2} (\mu_c-\mu_3)^{-1} (\mu_c-\mu_4)^{-1} \Phi_{ck}^{\hphantom{c}s} e_c(\mu_3) = \sum_{c=1}^{2} (\mu_c-\mu_3)^{-1} (\mu_c-\mu_4)^{-1} \Phi_{ck}^{\hphantom{c}s} e_c(\mu_4) = 0, \label{eq:Gauss2mix1} \\
  \sum_{c=1}^{2} (\mu_c-\mu_3)^{-1}(\mu_c-\mu_4)^{-1} (\Phi_{cj}^{\hphantom{c}r} \Phi_{ci}^{\hphantom{c}s} + \Phi_{ci}^{\hphantom{c}r} \Phi_{cj}^{\hphantom{c}s})= \delta_{ij}\delta_{rs} \Big(\kappa + \mu_3\mu_4 + \sum_{c=1}^{2} \phi_3^c \phi_4^c\Big). \label{eq:Gauss2mix2}
\end{gather}

Consider the set of $2$-dimensional vectors $((\mu_1-\mu_3)^{-1} (\mu_1-\mu_4)^{-1} \Phi_{1k}^{\hphantom{c}s}, (\mu_2-\mu_3)^{-1} (\mu_2-\mu_4)^{-1} \Phi_{2k}^{\hphantom{c}s})$, for all $k, s$ (such that $\la_k=\mu_3$ and $\la_s = \mu_4$). If this set contains two linearly independent vectors, then~\eqref{eq:Gauss2mix1} implies that $\mu_3$ and $\mu_4$ are locally constant, and then from equations~\eqref{eq:vnew2} we obtain that $\mu_1$ and $\mu_2$ are also locally constant, which is a contradiction.

Thus either $\Phi_{ck}^{\hphantom{c}s} = 0$, for all $c, k$ and $s$, or there exist a nonzero $p_3 \times p_4$ matrix $\Psi=(\Psi_k^s)$ and functions $\alpha_1, \alpha_2$, not simultaneously zeros, such that $\Phi_{ck}^{\hphantom{c}s} = \alpha_c \Psi_k^s$.

We show that the latter case leads to a contradiction. Indeed, in that case, equation~\eqref{eq:Gauss2mix2} gives $(\sum_{c=1}^{2}(\mu_c-\mu_3)^{-1} (\mu_c-\mu_4)^{-1} \alpha_c^2) (\Psi_j^r \Psi_i^s + \Psi_i^r \Psi_j^s) = \delta_{ij}\delta_{rs} \Big(\kappa + \mu_3\mu_4 + \sum_{c=1}^{2} \phi_3^c \phi_4^c\Big)$. If $\sum_{c=1}^{2}(\mu_c-\mu_3)^{-1} (\mu_c-\mu_4)^{-1} \alpha_c^2 \ne 0$, then taking $j=i, \, r \ne s$, with $\Psi_i^r \ne 0$ we obtain $\Psi_i^s = 0$. Then $\kappa + \mu_3\mu_4 + \sum_{c=1}^{2} \phi_3^c \phi_4^c=0$ (by taking $j=i, \, r=s$) which leads to a contradiction with taking $j=i, \, s=r$. Hence we obtain
\begin{equation}\label{eq:case2alpha}
(\mu_1-\mu_3)^{-1} (\mu_1-\mu_4)^{-1} \alpha_1^2 + (\mu_2-\mu_3)^{-1} (\mu_2-\mu_4)^{-1} \alpha_2^2 = 0.
\end{equation}
To avoid a contradiction, we need the coefficients of $\alpha_1^2$ and $\alpha_2^2$ in~\eqref{eq:case2alpha} to be of the opposite signs. Moreover, as we can simultaneously multiply $\alpha_1, \alpha_2$ and divide the matrix $\Psi=(\Psi_k^s)$ by some nonzero function, we can assume from~\eqref{eq:case2alpha} that $\alpha_1$ and $\alpha_2$ locally depend on (one of) the functions $\mu_m, \, m=1,2,3,4$. In particular, we have $e_i(\alpha_a) = e_r(\alpha_a) = 0$.

Furthermore, from~\eqref{eq:Gauss2mix1}, we obtain that the functions $\mu_3$ and $\mu_4$ (and hence also $\mu_1$ and $\mu_2$ in view of~\eqref{eq:vnew2}) are constant along the integral curves of the vector field $V=(\mu_1-\mu_3)^{-1} (\mu_1-\mu_4)^{-1} \alpha_1 e_1 + (\mu_2-\mu_3)^{-1} (\mu_2-\mu_4)^{-1} \alpha_2 e_2$. Hence the distribution $\br V \oplus L_3 \oplus L_4$ is integrable (its integral submanifolds are the level hypersurfaces of any of the function $\mu_m, \, m=1,2,3,4$). Using the fact that both $\mu_m$ and $\alpha_a$ are constant along the trajectories of both $e_i$ and $e_r$ we find that the vector fields $[V,e_i]$ and $[V,e_r]$ lie in $\br V \oplus L_3 \oplus L_4$ if and only if both vector fields $\eta_i((\mu_2-\mu_3)^{-1} (\mu_1-\mu_4)^{-1} \alpha_1 e_2 - (\mu_1-\mu_3)^{-1} (\mu_2-\mu_4)^{-1} \alpha_2 e_1)$ and $\xi_r((\mu_2-\mu_4)^{-1} (\mu_1-\mu_3)^{-1} \alpha_1 e_2 - (\mu_1-\mu_4)^{-1} (\mu_2-\mu_3)^{-1} \alpha_2 e_1)$ are the multiples of $V$. If $\xi_r \ne 0$ or $\eta_i \ne 0$, this gives $(\mu_2-\mu_3)(\mu_1-\mu_4) + (\mu_1-\mu_3) (\mu_2-\mu_4) = 0$, and an easy check shows that this equation, together with~\eqref{eq:vnew2}, implies that all four $\mu_m$'s are locally constant. 
It follows that $\xi_r = 0$ for all $r$ with $\la_r = \mu_4$, and $\eta_i = 0$ for all $i$ with $\la_i = \mu_3$. Then equations~\eqref{eq:Gauss2aija} and~\eqref{eq:Gauss2arsa} imply that both matrices $\Psi \Psi^t$ and $\Psi^t \Psi$ are the multiples of the identity matrices. As $\Psi \ne 0$, this is only possible when $p_3=p_4$ (and then $\Psi$ is a multiple of an orthogonal matrix). But if $p_3=p_4$, then equations~\eqref{eq:vnew2} imply $\mu_2=-\mu_1$ and $\mu_4=-\mu_3$, and so~\eqref{eq:case2alpha} gives $\alpha_1=\alpha_2=0$, a contradiction.

Thus $\Phi_{ck}^{\hphantom{c}s} = 0$, for all $c, k$ and $s$ (such that $\la_k=\mu_3, \, \la_s = \mu_4$ and $c \in \{1,2\}$). Then equations ~\eqref{eq:Gauss2aija} and~\eqref{eq:Gauss2arsa}, viewed as matrix equations, imply that a linear combination of the identity matrix of size $p_3 \times p_3$ (respectively, of size $p_4 \times p_4$) and a rank $1$ matrix is zero. As $p_3, p_4 > 1$, this is only possible when the linear combination is trivial, that is, when $\eta_i=0, \, \xi_r = 0$ and
\begin{equation} \label{eq:case2eaphia}
e_a(\phi_3^a) = \phi_3^b \gamma_a + (\phi_3^a)^2 + \mu_a\mu_3 + \kappa, \quad e_a(\phi_4^a) = \phi_4^b \gamma_a + (\phi_4^a)^2 + \mu_a\mu_4 + \kappa,
\end{equation}
where $\{a,b\}=\{1,2\}$. Moreover, equation~\eqref{eq:Gauss2mix2} and Gauss equations for $R_{i12i}$ and $R_{r12r}$ give, respectively,
\begin{gather}\label{eq:case2alg}
  \kappa + \mu_3\mu_4 + \phi_3^1 \phi_4^1 + \phi_3^2 \phi_4^2 = 0,\\
  e_a(\phi_3^b) = - \phi_3^a \gamma_a + \phi_3^1 \phi_3^2, \quad e_a(\phi_4^b) = - \phi_4^a \gamma_a + \phi_4^1 \phi_4^2, \label{eq:case2eaphib}
\end{gather}
where $\{a,b\}=\{1,2\}$. We show that these equations, together with~\eqref{eq:vnew2} will force all the $\mu_m, \, m=1,2,3,4$ to be locally constant, hence leading to a contradiction.

We first consider the case when $p_3=p_4$. Denote $l=2p_3-1$. From~\eqref{eq:vnew2} we get $\mu_2=-\mu_1, \; \mu_4=-\mu_3$ and $\mu_1^2 + l\mu_3^2 = 2 \kappa$, and then differentiating the latter equation and using~\eqref{eq:codazzi2n} we get $\phi_3^a = -2 l^{-1} \mu_1^2 (\mu_3 (\mu_3 -\mu_a))^{-1} \gamma_b, \; \phi_4^a = -2 l^{-1} \mu_1^2 (\mu_3 (\mu_3 + \mu_a))^{-1} \gamma_b$, where $\{a,b\}=\{1,2\}$. Substituting into~\eqref{eq:case2alg} we get $\gamma_1^2 + \gamma_2^2 = \frac14 l^2 \mu_1^{-4} \mu_3^2 (\mu_1^2-\mu_3^2) (\kappa-\mu_3^2)$. On the other hand, substituting $\phi_3^a$ and $\phi_4^a$ into~\eqref{eq:case2eaphia} and eliminating the derivatives of $\gamma_1$ and $\gamma_2$ we find
$\gamma_1^2 = \gamma_2^2 = -\frac{1}{8}l^2\mu_1^{-2}(\mu_1^2 + l\mu_3^2)^{-1}\mu_3^2(\mu_1^2-\mu_3^2)(\kappa-\mu_3^2)$.
Now eliminating $\gamma_1$ and $\gamma_2$ using the former equation, we obtain either $\kappa=\mu_3^2$ or $2\mu_1^2 + l \mu_3^2=0$. In both cases, when combined with $\mu_1^2 + l \mu_3^2 = 2\kappa$ we get a contradiction arises. 

Now suppose $p_3 \ne p_4$. Differentiating~\eqref{eq:vnew2}, from~\eqref{eq:codazzi2n} we obtain that both $\phi_3^a$ and $\phi_4^a$ equal $\gamma_b$ times some rational functions of the $\mu$'s (where $\{a,b\}=\{1,2\}$). As all the $\mu$'s are functionally dependent, we deduce that the $2$-dimensional vectors $(\phi_3^a e_1(\phi_4^a)-\phi_4^a e_1(\phi_3^a), \phi_3^a e_2(\phi_4^a)-\phi_4^a e_2(\phi_3^a))$ and $(e_1(\mu_1), e_2(\mu_1))$ are linearly dependent, for $a=1,2$. Using~\eqref{eq:case2eaphia} and~\eqref{eq:case2eaphib} (and the expressions for $\phi_3^a$ and $\phi_4^a$) we obtain two equations of the form $A \gamma_1^2 + B \gamma_2^2 +C = 0$, where $A,B$ and $C$ are certain polynomials in the $\mu$'s. The third equation of the same form comes from~\eqref{eq:case2alg}. Eliminating $\gamma_1^2$ and $\gamma_2^2$ from the resulting three equations we get a polynomial equation for the $\mu$'s. Together with~\eqref{eq:vnew2} this gives us four polynomial equations which appear to be symmetric relative to $\mu_1$ and $\mu_2$. Eliminating $\mu_3$ and $\mu_4$ we obtain two polynomial equations in $\sigma_1=\mu_1+\mu_2$ and $\sigma_2=\mu_1\mu_2$. Viewed as polynomials in $\sigma_2$, they have the leading terms $\big((p_3 + p_4 - 1) (p_3 + p_4)^{12} \sigma_1^{12}\big)\sigma_2^{12}$ and $\big((p_3 + p_4)^2 \sigma_1^2\big) \sigma_2^2$, respectively. Eliminating $\sigma_2$ we obtain a polynomial equation in $\sigma_1$ with constant coefficients whose constant term is $4096\kappa^{18} (p_3 + p_4 - 1) (p_3 - p_4)^{12}$. Hence $\sigma_1$ is locally a nonzero constant, and then $\sigma_2$ is also locally a constant. This implies that both $\mu_1$ and $\mu_2$ are locally constants, a contradiction. 

\subsection{\texorpdfstring{Case 3. $p_1=p_2=p_3=1<p_4$}{p{\textoneinferior} = p{\texttwoinferior} = p{\textthreeinferior} = 1 < p{\textfourinferior}}}
\label{ss:3p1}
Choose the labelling in such a way that $\la_1=\mu_1, \, \la_2=\mu_2, \, \la_3=\mu_3$, and $\la_i=\mu_4$ for $i>3$.

Note that $\mu_4 \ne 0$ everywhere (for otherwise equations~\eqref{eq:vnew} give $\mu_1\mu_2\mu_3=0$, which is a contradiction as the $\mu_a$'s are pairwise nonequal).

By Lemma~\ref{l:q4}\eqref{it:1implies4}, we have $e_i(\mu_a)=0$, for all $a=1,2,3,4$ and all $i > 3$. Moreover, the distribution $L_4$ is integrable, and its integral submanifolds are totally umbilical (or totally geodesic) in both $M^n$ and $\tM^{n+1}(\kappa)$. Codazzi equations~\eqref{eq:codazzi} give $\G_{11}^{\hphantom{1}i}=\G_{22}^{\hphantom{2}i}=\G_{33}^{\hphantom{2}i}=0$ for $i > 3$, and $\G_{ij}^{\hphantom{i}1}= \G_{ij}^{\hphantom{i}2}= \G_{ij}^{\hphantom{i}3}=0$ for all $i \ne j,\,  i, j > 3$. We also obtain $\G_{ii}^{\hphantom{i}a}=\G_{jj}^{\hphantom{j}a}$ for $i, j > 3, \, a=1,2,3$, and so we denote $\G_{ii}^{\hphantom{i}a}=\phi_a$, where $i> 3 \ge a$.
Let us put  $\gamma_b^a=\G_{bb}^{\hphantom{b}a}$, for $a, b \in \{1,2,3\}, \, a \ne b$. The remaining Codazzi equations are
\begin{gather}
e_a(\mu_4)=\phi_a(\mu_4-\mu_a), \qquad e_a(\mu_b)=\gamma_b^a(\mu_b-\mu_a), \label{eq:codazzi4diff12}\\
    \G_{ia}^{\hphantom{i}b}(\mu_a-\mu_b) = \G_{ab}^{\hphantom{1}i}(\mu_b-\mu_4), \label{eq:codazzi4diff35}\\
    \G_{12}^{\hphantom{1}3}(\mu_3-\mu_2)=\G_{23}^{\hphantom{1}1}(\mu_1-\mu_3)=\G_{32}^{\hphantom{1}1}(\mu_1-\mu_2)=:\delta, \label{eq:codazzi4diff6}
\end{gather}
where $a, b \in \{1,2,3\}, \, a \ne b$, and $i>3$. The equation \eqref{eq:codazzi4diff6} can be written as $\G_{cb}^{\hphantom{b}a} = (\mu_a-\mu_b)^{-1}\delta$ for $\{a,b,c\}=\{1,2,3\}$. Since $\mu_4=-\mu_1-\mu_2-\mu_3$ by~\eqref{eq:vnew}, from \eqref{eq:codazzi4diff12} we obtain
\begin{equation}\label{eq:codazzi4diff7}
  e_a(\mu_a)=\phi_a (2 \mu_a + \mu_b + \mu_c) - \gamma_b^a (\mu_b-\mu_a) -\gamma_c^a(\mu_c-\mu_a), \quad \{a,b,c\}=\{1,2,3\}.
\end{equation}

We next show that $\G_{ka}^{\hphantom{k}b}=0$ for all $k>3$ and all $a,b \le 3$.

Computing $R_{kaak}$ for $a\le 3<k$ from Gauss equations~\eqref{eq:gauss} we obtain that the expression $(\G_{k1}^{\hphantom{k}2})^2 \frac{\mu_1-\mu_2}{\mu_4-\mu_2}-(\G_{k3}^{\hphantom{k}1})^2\frac{\mu_3-\mu_1}{\mu_4-\mu_3}$ and the other two obtained from it by cyclicly permutating the triple $(1, 2, 3)$ do not depend on $k>3$. Define the vectors $v_1, v_2, v_3\in L_4$ by $\<v_c, X\>=\<\nabla_X e_a, e_b\>$, where the triple $(a,b,c)$ is a cyclic permutating the triple $(1, 2, 3)$. Since the choice of an orthonormal basis $\{e_k\},\, 3 \le k \le n$, for $L_4$ is arbitrary, we obtain that the quadratic forms $Q_a$ on $L_4$ defined by $Q_a(X)= \<v_c, X\>^2 \frac{\mu_a - \mu_b}{\mu_4 - \mu_b} + \<v_b,X\>^2 \frac{\mu_a - \mu_c}{\mu_4 - \mu_c}$, where $\{a,b,c\}=\{1,2,3\}$, are constant on the unit sphere of $L_4$, and hence $Q_a(X)=C_a \|X\|^2$ for all $X \in L_4$, with some constants $C_1, C_2$ and $C_3$. We consider two cases.

First suppose that $p_4 \ge 3$. Then we can take a nonzero $X \perp v_b, v_c$ to obtain $C_a = 0$, and hence $Q_a=0$ for $a=1,2,3$. Without loss of generality, we can assume that the points $\mu_2$ and $\mu_3$ do not lie on the interval between the points $\mu_1$ and $\mu_4$ on the real line. Then the cross-ratio $(\mu_1, \mu_4; \mu_2, \mu_3)=\frac{(\mu_2 - \mu_1) (\mu_3 - \mu_4)}{(\mu_2 - \mu_4) (\mu_3 - \mu_1)}$ is positive, and so equation $Q_1=0$ implies that $v_2=v_3=0$. Then from $Q_2=0$ we get $v_1=0$. Hence $\G_{ka}^{\hphantom{k}b}= 0$, for all $k>3$ and all $a,b \le 3$, as required.

A little more work is needed when $p_4=2$ (note that $p_4 \ge 2$ as $n \ge 5$). Out of the three cross-ratios $k_{a}=(\mu_a, \mu_4; \mu_b, \mu_c)=\frac{(\mu_b - \mu_a) (\mu_c - \mu_4)}{(\mu_b - \mu_4) (\mu_c - \mu_a)}$, where $\{a,b,c\}=\{1,2,3\}$, exactly one is negative, say $k_1 < 0 < k_2,k_3$. Then $Q_1$ is non-positive definite on $\br^2 = L_4$, and so the equation $Q_1(X)=C_1\|X\|^2$ implies $Q_1=0$, hence $v_2$ and $v_3$ are linear dependent. Then from the equations $Q_2(X)=C_2\|X\|^2$ and $Q_3(X)=C_3\|X\|^2$ we deduce that $v_1 \perp v_2, v_3$. Relative to a basis $\{e_4, e_5\}$ for $L_4$ such that $v_1$ is a multiple of $e_4$ and $v_2,v_3$ are multiples of $e_5$, equations $Q_2(X)=C_2\|X\|^2,\, Q_3(X)=C_3\|X\|^2$ and $Q_1=0$ give
\begin{equation*}
\G_{41}^{\hphantom{k}2} =\G_{43}^{\hphantom{k}1} = \G_{52}^{\hphantom{k}3} =0, \qquad
\G_{51}^{\hphantom{k}2}\mu_{12}=\G_{51}^{\hphantom{k}3}\mu_{13}=\G_{42}^{\hphantom{k}3}\mu_{23}=:\theta,
\end{equation*}
where $\mu_{ab}:=\sqrt{|(\mu_4-\mu_c)(\mu_a-\mu_b)|}$ for $\{a,b,c\}=\{1,2,3\}$ (up to changing the sign of some of the vectors $e_2$ and $e_4$). But then from Gauss equations~\eqref{eq:gauss} for $R_{4125}$ and $R_{4135}$, we obtain $e_4(\theta)=-\theta^2 (\mu_{12}\mu_{23})^{-1} \mu_{13} k_1 = \theta^2 (\mu_{13}\mu_{23})^{-1} \mu_{12} k_1^{-1}$, which implies $\theta = 0$, as required.

The rest of the proof is straightforward, but somewhat computationally involved.

As $\G_{ka}^{\hphantom{k}b}= 0$, equation \eqref{eq:codazzi4diff35} implies that $\G_{ab}^{\hphantom{b}k}=0$, for all $k > 3$ and all $a, b \le 3$. It follows that the distribution $L_1 \oplus L_2 \oplus L_3$ is integrable and that its integral submanifolds are totally geodesic in $M^n$. We can now specify the vector fields $e_i,\, i > 3$, choosing them arbitrarily on an integral submanifold of the distribution $L_4$, and then parallelly translating along totally geodesic integral submanifolds of the distribution $L_1 \oplus L_2 \oplus L_3$. Then we obtain $\G_{ai}^{\hphantom{1}j}=0$, for all $i, j > 3 \ge a$.

Gauss equations~\eqref{eq:gauss} for $R_{kaak}, \, R_{kabk}, \, R_{abba}$ and $R_{abbc}$, where $\{a,b,c\}=\{1,2,3\}$ and $k>3$ give, respectively,
\begin{gather}
    e_a(\phi_a) - \sum_{b\neq a}\gamma_{a}^{b}\phi_b-\phi_a^2 = \kappa +\mu_a\mu_4, \qquad
    e_a(\phi_b) = \phi_a \phi_b + \phi_c \G_{ab}^{\hphantom{1}c} - \gamma_{a}^{b}\phi_a,\label{eq:RkabkRkaak}\\
    e_a(\gamma_{b}^{a}) + e_b(\gamma_{a}^{b}) - (\gamma_{a}^{b})^2-(\gamma_{b}^{a})^2 + \G_{ab}^{\hphantom{1}c}\G_{ba}^{\hphantom{1}c}-\G_{ca}^{\hphantom{1}b}\G_{ac}^{\hphantom{1}b}-
    \G_{cb}^{\hphantom{1}a}\G_{bc}^{\hphantom{1}a}-\kappa-\mu_a\mu_b=0,\label{eq:Rabba}\\
    e_a(\gamma_{b}^{c}) - e_b(\G_{ab}^{\hphantom{1}c}) + \gamma_{b}^{a}\gamma_{a}^{c} + \gamma_{a}^{b}(\G_{ba}^{\hphantom{1}c}+\G_{ab}^{\hphantom{1}c}) - \gamma_{b}^{c}\gamma_{b}^{a} + \gamma_{c}^{b}(\G_{ab}^{\hphantom{1}c}-\G_{ba}^{\hphantom{1}c})=0.\label{eq:Rabbc}
\end{gather}

Denote $m=n-3$. From~\eqref{eq:vnew} we get $\mu_4=-\mu_1-\mu_2-\mu_3$ and then $f_2=f_3=0$ yield
\begin{gather*}
    f_2=(\mu_1^2+\mu_2^2+\mu_3^2)+(2m-1)(\mu_1+\mu_2+\mu_3)^2-4\kappa=0,\\
f_3=\mu_1\mu_2\mu_3 - m(\mu_1+\mu_2+\mu_3)((\mu_1+\mu_2+\mu_3)^2 - 2\kappa)=0.
\end{gather*}
We note that these equations give $\mu_1\mu_2+\mu_1\mu_3+\mu_2\mu_3=m\sigma^2 -2\kappa$ and $\mu_1\mu_2\mu_3=m\sigma(\sigma^2-2\kappa)$, where $\sigma = \mu_1 + \mu_2 + \mu_3$, and so all four function $\mu_i$ are locally functions of $\sigma$.

Differentiating $f_2$ and $f_3$ along $e_a$ and using \eqref{eq:codazzi4diff12} we can solve for $\gamma_{b}^{a}$ to obtain $\gamma_{b}^{a}=\phi_a P_b^a(\mu)$, where $P_b^a(\mu)$ are rational functions of $\mu_1,\mu_2,\mu_3$ and $a,b \le 3, \, a \ne b$. 

Next we substitute these expressions for $\gamma_{b}^{a}$ into the equations~\eqref{eq:Rabba}, and then use~\eqref{eq:RkabkRkaak} for the derivatives of $\phi_a$ and the equation $\G_{cb}^{\hphantom{b}a} = (\mu_a-\mu_b)^{-1}\delta$ which follows from \eqref{eq:codazzi4diff6} to obtain three polynomial equations $h_{ab}=0$, where $1 \le a < b \le 3$. Every polynomial $h_{ab}$ is a linear combination of the functions $\Phi_c = (\phi_c)^2, \, \delta^2$, and $1$, with the coefficients depending only on the $\mu_c, \, c =1,2,3$.

We now consider two equations: the equation for $R_{1223}$ given by~\eqref{eq:Rabba} and the derivative of $h_{13}$ along $e_2$. Substituting the expressions for $\gamma_{b}^{a}$ into them and then using~\eqref{eq:RkabkRkaak} and \eqref{eq:codazzi4diff6} we can eliminate $e_2(\delta)$ to  obtain a polynomial $f_4$ in the $\phi_a$'s, $\mu_b$'s and $\delta$. The resultant of $h_{13}$ and $f_4$ relative to $\delta$ gives a polynomial equation $f_5=0$, where $f_5$ is of the form $4(\mu_1-\mu_2)^2 (\mu_2-\mu_3)^2 (\mu_1-\mu_3)^6 f_6$, and $f_6$ is a polynomial in $\Phi_c = (\phi_c)^2$, with coefficients depending only on the $\mu$'s.

Eliminating $\delta^2$ from the system $h_{12}=h_{13}=h_{23}=0$, we obtain two polynomials $f_7$ and $f_8$ in the $\Phi_c$'s, with the coefficients only depending on the $\mu_a$'s. Solving $f_7=f_8=0$ for $\Phi_1$ and $\Phi_3$ in terms of $\Phi_2$ and the $\mu_a$'s, and substituting the resulting expressions into $f_6$, we obtain a polynomial equation $f_9=0$ for $\Phi_2$ with coefficients depending only on the $\mu_a$'s whose leading term is $-4(\mu_1-\mu_2)^2(\mu_2-\mu_3)^2(\mu_1+2\mu_2+\mu_3)^2 f_{10} \Phi_2^{3}$, where $f_{10}$ is a polynomial in the $\mu_a$'s symmetric by $\mu_{1}$ and $\mu_{3}$, but not by all three $\mu_a$'s. The polynomial $f_{10}$ is nonzero modulo the ideal generated by $f_2$ and $f_3$ (to see this we take its symmetrisation by the permutation group on the elements $(\mu_1,\mu_2,\mu_3)$ and then use the equations $f_2=f_3=0$ to express the resulting symmetric polynomial in terms of $\sigma$; the leading term of it $\sigma^{60}$ times a polynomial with constant coefficients in $m$ having no integer roots $m \ge 2$). It follows that locally $\Phi_2$ (and hence also $\phi_2$) is a function of the $\mu_a$'s (that is, a function of $\sigma$). By symmetry, the same is true for $\phi_1$ and $\phi_3$. Therefore, we obtain $e_{a}(\phi_c) e_b(\sigma) = e_b(\phi_c) e_a(\sigma)$ for all $a,b,c\in\{1,2,3\}$. 

We now eliminate the case when one of the functions $\phi_a$ is locally zero. Suppose that locally $\phi_1 = 0$. Then the Gauss equation~\eqref{eq:gauss} for $R_{akk1}=0$ with $a=2,3$ and $k>3$, gives $\phi_a \delta = 0$, implying that $\delta=0$ (as otherwise, $\phi_2=\phi_3=0$, causing all the $\gamma_{b}^{a}$ to vanish, and thus the $\mu_c$'s to be locally constant). Substituting $\phi_1=\delta=0$ into the polynomials $h_{12}, h_{13}$ and $h_{23}$ we obtain three linear inhomogeneous equations for the vector $(\Phi_2,\Phi_3,1)^{t}$.
Its determinant is given by $(\mu_1+\mu_2+2\mu_3)(\mu_1+2\mu_2+\mu_3)(\mu_1-\mu_2)^3(\mu_1-\mu_3)^3(\mu_2-\mu_3)^4 F(\mu_1,\mu_2,\mu_3)$, where $F$ is a symmetric polynomial in $\mu_2, \mu_3$ (recall that by \eqref{eq:vnew} we have $\mu_1+\mu_2+2\mu_3 = \mu_3-\mu_4$ and $\mu_1+2\mu_2+\mu_3 = \mu_2-\mu_4$). If $F$ were an element of the ideal generated by $f_2$ and $f_3$, then its symmetrisation $F_s$ by $\mu_1,\mu_2,\mu_3$ also would. But using the equations $f_2=f_3=0$ to express $F_s$ in terms of $\sigma$ we obtain a polynomial whose leading term equals $\kappa \sigma^{18}$ times a polynomial with constant coefficients in $m$ having no integer roots $m \ge 2$. It would then follow that $\sigma$ is locally a constant, and hence all the $\mu_i$'s are locally constant leading to a contradiction. Therefore, all the $\phi_a$'s are locally functions of $\sigma$ and are locally nonzero. 

Now solving the equation $e_1(\phi_1)e_2(\sigma) - e_2(\phi_1)e_1(\sigma) = 0$ for $\delta$, we obtain that $\delta\phi_1\phi_2\phi_3$ is given by a polynomial $G$ in the $\Phi_c$'s, with the coefficients being rational functions of the $\mu_a$'s (whose denominators are the products of terms of the form $(\mu_a-\mu_b), \, 1\le a < b \le 4$, and hence are nonzero). Letting $G_1$ be the function obtained by interchanging the indices $1$ and $3$ in $G$, we get the equation $G_1-G=0$. Its left-hand side is a rational function whose denominator is a nonzero polynomial in the $\mu_a$'s and whose numerator has the form $\Phi_2 H$, where $H$ is a polynomial of degree $1$ in the $\Phi_c$'s with coefficients being polynomials in the $\mu_a$'s. Letting $H_1$ be the polynomial obtained by switching indices $1$ and $2$ in $H$, we have obtain $H_1-H = (\mu_1+\mu_2+2\mu_3)(\mu_1-\mu_2)^3(\mu_1-\mu_3)^2(\mu_2-\mu_3)^2((\mu_1+\mu_2+\mu_3)^2  + \Phi_1 + \Phi_2 +\Phi_3 + \kappa)$, which cannot be zero (recall that $\mu_1+\mu_2+2\mu_3 = \mu_3-\mu_4$ and $\kappa > 0$).

This contradiction completes the proof of the fact that a hypersurface $M^n \subset \tM^{n+1}(\kappa)$ satisfying the assumptions of Theorem~\ref{th:main} cannot locally have four principal curvatures.

\section{Proof of Theorem~\ref{th:main}}
\label{s:proof}

Let $M^n \subset \tM^{n+1}(\kappa)$ be a hypersurface satisfying the assumptions of Theorem~\ref{th:main}. By Lemma~\ref{l:four}, at every point of $M^n$ the number $q$ of different principal curvatures is at most $4$. If at some point we have $q=4$, then the same is true in a neighbourhood of that point which leads to contradiction with the result of Section~\ref{s:q=4}. It follows that at every point we have at most $3$ principal curvatures. But if $q=3$ at some point, then $q=3$ also in a neighbourhood of that point contradicting the result of Section~\ref{s:q=3}. Hence at every point of $M^n$, we have $q \le 2$, and moreover, there exists at least one point at which $q = 2$, as otherwise $M^n$ is totally umbilical and hence Einstein (is of constant curvature).

The set $\mU_u \subset M^n$ of umbilical points is closed. At every point $x$ of its complement, $M^n$ has two principal curvatures, $\mu_1(x)$ and $\mu_2(x)$, with multiplicities $p_1(x)$ and $p_2(x)$ respectively. We define $m(x) = \min(p_1(x), p_2(x))$. Then $M^n \setminus \mU_u$ is the nonempty union of pairwise disjoint, open subsets $\mU_m \subset M^n, \, 1 \le m \le n/2$, where $\mU_m=\{x \in M^n \, | \, m(x)=m\}$.

We first suppose for some $m \ge 2$, the set $\mU_m$ is nonempty, and let $\mU_m'$ be one of its connected components. Then both multiplicities $p_1$ and $p_2$ and both principal curvatures $\mu_1$ and $\mu_2$ are constant on $\mU_m'$. By the argument in Section~\ref{s:q=2}, the hypersurface $\mU_m'$ is a domain of the product of two spaces of constant curvature, either as given in Example~\ref{ex:prod}, or both Einstein and weakly Einstein, the only possibility for which is a domain on the product of two congruent spheres (see Section~\ref{s:intro}). But as principal curvatures are constant, the subset $\mU_m'$ is not only open, but also closed. It follows that $\mU_m' = M^n$, and as $M^n$ is not allowed to be Einstein, we obtain the hypersurfaces as in Theorem~\ref{th:main}\eqref{it:thprod}.

Now suppose that the all the sets $\mU_m$ with $m \ge 2$ are empty. Then $\mU_1$ must be nonempty (and open), but not necessarily connected. On $\mU_1$, we have two principal curvatures: $\mu_1$ of multiplicity $1$ and $\mu_2$ of multiplicity $n-1$ (these are well-defined, as $n \ge 3$). Recall that by~\eqref{eq:q=2}, the weakly Einstein condition is equivalent to the equation $\mu_2(\mu_2^2+\mu_2 \mu_1 + 2 \kappa) = 0$. Let $\mU_d = \{x \in \mU_1 \, | \,\mu_2(x) = 0\}$ and $\mU_w = \mU_1 \setminus \mU_d$. The set $\mU_w$ is open and consists of the points of $M^n$ at which the Ricci tensor is not a multiple of the metric tensor. At all points of $\mU_w$ we have $\mu_2^2+\mu_2 \mu_1 + 2 \kappa = 0$.

We have the disjoint union decomposition $M^n = \mU_u \cup \mU_d \cup \mU_w$, with both $\mU_u$ and $\mU_u \cup \mU_d$ being closed ($\mU_d$ by itself may be neither open, nor closed). Suppose $M^n \ne \mU_w$, that is, the boundary $\partial \mU_w$ is nonempty. Let $y \in \partial \mU_w$. Then $y \notin \mU_w$ (as $\mU_w$ is open) and $y \notin \mU_d$. To see the latter, we note that the function $\mu_2^2+\mu_2 \mu_1$ is well-defined and continuous on $\mU_1= \mU_d \cup \mU_w$, but equals $-2\kappa$ on $\mU_w$, and zero at $y$. It follows that $y \in \mU_u$ (so that $y$ is an umbilical point). This cannot occur when $\kappa > 0$. Indeed, the function $\|S-H\id\|^2$ is continuous on $M^n$, and at the points of $\mU_w$, we have $\|S-H\id\|^2 = \frac{n-1}{n}(\mu_1-\mu_2)^2$. But on $\mU_w$ we have $|\mu_2-\mu_1| = 2(|\mu_2|^{-1} \kappa + |\mu_2|) \ge 4 \sqrt{\kappa}$, and so $\|S-H\id\|^2 \ge 16 \frac{n-1}{n} \kappa$, while at an umbilical point, $\|S-H\id\|^2 = 0$. Suppose $\kappa = -1$. Let $\mU_w'$ be a connected component of $\mU_w$ having the point $y \in \mU_u$ on its boundary. Then we can choose Cartesian coordinates $\{x_1, \dots, x_{n+2}\}$ in the Minkowski space $\br^{n+2,1}$ in such a way that $\mU_w'$ is a domain on one of the rotation hypersurfaces given in Example~\ref{ex:rotation}\eqref{it:rothh},~\eqref{it:roths},~\eqref{it:rothp} (see Section~\ref{s:q=2}). In cases \eqref{it:rothh} and \eqref{it:roths} we have $x_1^2=\delta s^2 + b$, where $\delta = \pm 1$, and so by~\cite[Proposition~3.2]{dCD}, $\mu_2^2 = -\kappa + \delta b x_1^{-4}$ and $\mu_1 \mu_2 = -\kappa - \delta b x_1^{-4}$ (note that $b \ne 0$, in both cases, for otherwise $\mu_1=\mu_2$). But then for the points of $\mU_w'$ lying in a small neighbourhood $\mU(y)$, we have $|\mu_2^2 - \mu_1 \mu_2| > \varepsilon$ for some $\varepsilon > 0$, and so $|\mu_2 - \mu_1| > \varepsilon'$ for some $\varepsilon' > 0$ (for otherwise $\mu_2$ would be unbounded on $\mU_w' \cap \mU(y)$, and so also would be $H$). It follows that on the open set $\mU_w' \cap \mU(y)$ we have $\|S-H\id\|^2 > \frac{n-1}{n}(\varepsilon')^2$, and hence the point $y$ lying on its boundary cannot be umbilical. Similar argument also applies in case~\eqref{it:rothp}. Then $x_1^2=as, \, a \ne 0$, and then by~\cite[Proposition~3.2]{dCD} we get $\mu_2^2 = -\kappa - \frac14 a^2 x_1^{-4}$ and $\mu_1 \mu_2 = -\kappa + \frac14 a^2 b x_1^{-4}$ on $\mU_w' \cap \mU(y)$. Then the function $|\mu_2^2 - \mu_1 \mu_2|$ is bounded away from zero on $\mU_w' \cap \mU(y)$, as also is $|\mu_2 - \mu_1|$, and hence also is $\|S-H\id\|^2$. Therefore the point $y$ cannot be umbilical.

It follows that the open subset $\mU_w \subset M^n$ has empty boundary, and so $\mU_w = M^n$. Then $M^n$ has two principal curvatures everywhere, and no points at which the Ricci tensor is a multiple of the metric tensor. Now the classification in Section~\ref{s:q=2} completes the proof of Theorem~\ref{th:main}.

\section*{Declaration}

\noindent \textbf{Funding.} The first and third named authors were supported by the National Research Foundation of Korea(NRF) grant funded by the Korea government(MSIT) (RS-2024-00334956). The second named author was partially supported by ARC Discovery grant DP210100951.

\noindent \textbf{Competing Interests.} The authors have no relevant financial or non-financial interests to disclose.

\noindent \textbf{Author Contributions.} All authors read and approved the final manuscript.

\noindent \textbf{Data Availability.} The manuscript has no associated data.

\bigskip

\end{document}